\documentclass[12pt]{article}

\usepackage{pifont}
\usepackage{amsfonts}
\usepackage{amsmath}
\usepackage{amssymb}
\usepackage{amsmath,bm}
\usepackage{amsthm}
\usepackage{appendix}
\usepackage{lineno,hyperref}
\usepackage{mathtools}
\usepackage{theoremref}
\usepackage{xcolor}
\modulolinenumbers[5]

\newtheorem{theorem}{Theorem}[section]   
\newtheorem{assumption}{Assumption}[section]
\newtheorem{corollary}{Corollary}[section]
\newtheorem{proposition}{Proposition}[section]
\newtheorem{lemma}{Lemma}[section]
\newtheorem{definition}{Definition}[section]

\newtheorem{remark}{Remark}[section]

\newtheorem{conjecture}{Conjucture}[section]

\newtheorem{hypothesisinner}{}
\newenvironment{hypothesis}{\hypothesisinner}{\endhypothesisinner}

\newcommand{\ZZ}{\mathbb{Z}} %integer
\newcommand{\RR}{\mathbb{R}} %real number
\newcommand{\Acal}{\mathcal{A}}

\newcommand{\Lcal}{\mathcal{L}}
\newcommand{\Mcal}{\mathcal{M}}

\def\Xint#1{\mathchoice
{\XXint\displaystyle\textstyle{#1}}%
{\XXint\textstyle\scriptstyle{#1}}%
{\XXint\scriptstyle\scriptscriptstyle{#1}}%
{\XXint\scriptscriptstyle\scriptscriptstyle{#1}}%
\!\int}
\def\XXint#1#2#3{{\setbox0=\hbox{$#1{#2#3}{\int}$ }
\vcenter{\hbox{$#2#3$ }}\kern-.6\wd0}}

\def\epsilon{\varepsilon}

\title{\textbf{\huge{Generalized principal eigenvalues of elliptic operators and spreading speeds of Fisher-KPP equations
in two-scale almost periodic media}}}
\author{Xing Liang \thanks{School of Mathematical Sciences, University of Science and Technology of China, Hefei, Anhui 230026, China,
\textit{Email address}: \texttt{xliang@ustc.edu.cn}} 
\quad Linfeng Xu \thanks{School of Mathematical Sciences, University of Science and Technology of China, Hefei, Anhui 230026, China,
\textit{Email address}: \texttt{xlf\_hzfy@mail.ustc.edu.cn}}
\quad Tao Zhou \thanks{Center for Pure Mathematics, School of Mathematical Science, Anhui University, Hefei, Anhui 230601, China,
\textit{Email address}: \texttt{tzhou910@ustc.edu.cn}}}
\date{} % delete the data

\begin{document}
%%%%%%%%%%%%%%%%%%%%%%%%%%%%%%%%%%%%%%%%%%%%%%%%%%%%%%%%%%%%%%%%
\newcommand{\supercite}[1]{\textsuperscript{\cite{#1}}}
%%%%%%%%%%%%%%%%%%%%%%%%%%%%%%%%%%%%%%%%%%%%%%%%%%%%%%%%%%%%%%%%
\maketitle

%%%%%%%%%%%%%%%%%%%%%%%%%%%%%%%%%%%%%%%%%%%%%%%%%%%%%%%%%%%%%%%%
%%%%%%%%%%%%%%%%%%%%%%%%%%%%%%%%%%%%%%%%%%%%%%%%%%%%%%%%%%%%%%%%
\setlength{\oddsidemargin}{ 1cm} % 3.17cm - 1 inch
\setlength{\evensidemargin}{\oddsidemargin}
\setlength{\textwidth}{13.50cm}
\vspace{-0.2cm}
\begin{center} 
\parbox{\textwidth}{ 
{\textbf{Abstract: }} This paper is concerned with the asymptotic behavior of the generalized principal eigenvalues
of elliptic operators and spreading speeds of Fisher-KPP equations in two-scale almost
periodic media where one scale is fixed and another one approaches zero or infinity.  
We transform the problem into the homogenization of certain effective Hamiltonian and then establish the asymptotic limits and the convergence rates.
Based on the analysis of the asymptotic behavior of effective Hamiltonians,  
we investigate how the heterogeneity of the advection and growth rates affect on the propagation in the case where the media has very rapid or slow spatial oscillation:
We show a normal scale perturbation of the growth rate with mean zero can accelerate the propagation in the media with rapid or slow oscillation; and an advection with slow oscillation and mean zero can decelerate the propagation in 1-D case.

} 
\end{center}

\vspace{0.5cm}
%%%%%%%%%%%%%%%%%%%%%%%%%%%%%%%%%%%%%%%%%%%%%%%%%%%%%%%%%%%%%%%%
%%%%%%%%%%%%%%%%%%%%%%%%%%%%%%%%%%%%%%%%%%%%%%%%%%%%%%%%%%%%%%%%
%\begin{equation*}
%\large{\textbf{chinese latex template}}\
%	
%%writer and communication address
%\textbf{author1, author2}
%	
%	
%\parbox{\textwidth}{
%%English abstract
%		
%\small{\textbf{Abstract}\quad This is a basic chinese latex template for 
%novice,could be used for course eassy\
%			
%					
%%English key word
%\textbf{Key Words}\quad latex, basic, chinese template, course eassy}}
%\end {center}

%%%%%%%%%%%%%%%%%%%%%%%%%%%%%%%%%%%%%%%%%%%%%%%%%%%%%%%%%%%%%%%%
%%%%%%%%%%%%%%%%%%%%%%%%%%%%%%%%%%%%%%%%%%%%%%%%%%%%%%%%%%%%%%%%
%%%%%%%%%%%%%%%%%%%%%%%%%%%%%%%%%%%%%%%%%%%%%%%%%%%%%%%%%%%%%%%%
%%%%%%%%%%%%%%%%%%%%%%%%%%%%%%%%%%%%%%%%%%%%%%%%%%%%%%%%%%%%%%%%
\setlength{\oddsidemargin}{-.5cm} % 3.17cm - 1 inch
\setlength{\evensidemargin}{\oddsidemargin}
\setlength{\textwidth}{17.00cm}

\counterwithin*{equation}{section}
\renewcommand\theequation{\thesection.\arabic{equation}}

\section{Introduction}
In this paper, we mainly consider the spreading properties of the following reaction-diffusion equation
\begin{equation}\label{equation}
\begin{cases}
u_t - \sum\limits_{i,j=1}^{n}a_L^{ij}(x)D_{ij}u - \sum\limits_{i=1}^{n}b_L^i(x) D_iu = (c_L(x) + \tilde{c}(x))u(1-u), \\
u(0,x) = u_0(x)\in C_c(\RR^n,[0,1]), \ \not \equiv 0 .
\end{cases}
\end{equation}
We always consider $L>0$ as a scale parameter and 
$$a_L^{ij}(x) = a^{ij}(x/L), \quad b_L^i(x) = b^i(x/L), \quad c_L(x) = c(x/L), \quad i,j=1,\cdots,n.$$ 
Denote $A_L(x)=A(x/L)$ with $A(x) = (a^{ij}(x))_{n \times n}$ and $B_L(x)=B(x/L)$ with $B(x) = (b^i(x))_{i=1}^n$.
In the one-dimensional case,
we simplify the notation by dropping the superscript and denoting the coefficients as $a$ and $b$, respectively.

Throughout this paper, we assume that the following hypotheses hold:

\begin{hypothesis}\thlabel{uniformelliptic}
$A(x)$ is symmetric and 
\begin{equation*}
\alpha_m |\xi|^2 \leq \sum\limits_{i,j=1}^{n}a^{ij}(x)\xi_i\xi_j \leq \alpha_M |\xi|^2
\end{equation*}
for all $x,\xi \in \RR^n$ and  some $0 < \alpha_m \leq \alpha_M < \infty$.
\end{hypothesis} 

\begin{hypothesis}\thlabel{almostperiodicity}
$a^{ij},b^i,c$ and $\tilde{c}$ are in $C^{2,\gamma}(\RR^n)$ for some $\gamma \in (0,1)$ with bounded derivatives;  
Furthermore, $a^{ij},b^i,c$ and $\tilde{c}$ are almost periodic  in the sense of Bohr:
  
\begin{definition} \cite{bohr1925theorie}
A function $f: \RR^n \to \RR$ is almost periodic if it is the uniform limit of a sequence of trigonometric polynomials in $\RR^n$.
\end{definition}
\end{hypothesis}

There is another equivalent definition by means of the existence of a relatively dense set of $\epsilon$ almost-periods for all $\epsilon > 0$. That is,

\begin{definition}\cite{shubin1978almost}
A function $f: \RR^n \to \RR$ is almost periodic if it is continuous and for any $\epsilon > 0$,
$$\{x \in \RR^n :  \|f(\cdot + x) - f(\cdot)\|_{L^{\infty}(\RR^n)} < \epsilon\}$$
is relatively dense.
In other words, there exists $M > 0$ and $\{x_n\}_{n \in \ZZ} \subset \RR^n$ such that
$$\|f(\cdot + x_n) - f(\cdot)\|_{L^{\infty}(\RR^n)} < \epsilon$$
with
$$\mathop{\bigcup}\limits_{n \in \ZZ}B(x_n,M) = \RR^n,$$
where $B(x,r)$ denotes the ball centered at $x$ with radius $r$.
\end{definition}

The study of the spreading properties of \eqref{equation} started from the celebrated work of Fisher \cite{fisher1937wave} and Kolmogorov, Petrovsky and Piskunov \cite{kolmogorov1937investigation}. 
Subsequently, 
this model has been widely applied to study species invasion in ecology \cite{skellam1951random,murray2007mathematical,cantrell2004spatial,shigesada1997biological,kallen1985simple}.
A central question for such model concerns how the stable steady state $1$ invades the unstable state $0$.
In the homogeneous case, Aronson and Weinberger in \cite{aronson1978multidimensional} showed that if $B = 0$ and $c + \tilde{c} > 0$, then there exists $\omega^* > 0$ such that for any solution $u(t,x)$ of \eqref{equation}, 
$$\left\{\begin{array}{ll}
\lim\limits_{t \to +\infty}\sup\limits_{|x| \leq \omega t}|u(t,x) - 1| = 0  \quad &\hbox{ for all } \omega \in (0,\omega^*), \\
\lim\limits_{t \to +\infty}\sup\limits_{|x| \geq \omega t}|u(t,x)| = 0 \quad &\hbox{ for all } \omega > \omega^*.
\end{array}\right.$$
This implies that for any direction $e \in \mathbb{S}^{n-1}$, 
an observer moving with speed $\omega > \omega^*$ will eventually observe the unstable state,
whereas if moving with speed $\omega \in (0,\omega^*)$,
they will encounter the stable state.
Therefore, these results characterize the spatial-temporal dynamics of the invasion and we prefer to call $\omega^*$ the {\bf spreading speed in direction $e$}.

Equation \eqref{equation} is a natural extension of the classical Fisher-KPP model to heterogeneous environments.
A large body of literature has been devoted to understanding the effect of the spatial heterogeneity on the spreading properties of the solutions.
In a heterogeneous medium,
the spreading speed generally depends on the direction of propagation.
In periodic settings under certain conditions on the coefficients,
the existence of an asymptotic spreading speed $\omega^*(e) > 0$ for each direction $e \in \mathbb{S}^{n-1}$ has been established \cite{berestycki2008asymptotic,gartner1979propagation,weinberger2002spreading,freidlin2020wavefront}. 
These speeds are defined by the following property:
for any compactly supported initial datum 
$u_0 \not \equiv 0$ with $0 \leq u_0 \leq 1$,
the solution satisfies 
\begin{equation*}
\begin{cases}
\liminf\limits_{t \to +\infty}u(t,x + \omega t e) = 1 \quad 
&\hbox{ if } 0 \leq \omega < \omega^*(e), \\
\lim\limits_{t \to +\infty}u(t,x + \omega t e) = 0 \quad 
&\hbox{ if } \omega > \omega^*(e).
\end{cases}
\end{equation*}
locally in $x \in \RR^n$.
\cite{nadin2009traveling,weinberger2002spreading} investigate those solutions with front-like initial data.
For general heterogeneous media, we refer to  \cite{berestycki2008asymptotic,berestycki2012spreading}.

One of our aims is to investigate the asymptotic behavior of the spreading speeds as $L \to +\infty$ and $L \to 0$, respectively.
In the periodic case, the asymptotic behavior of the spreading speeds of \eqref{equation} as $L \to +\infty$ was established by Hamel \textit{et al} \cite{hamel2010spreading,hamel2011viscosity} when $\tilde{c} = 0$. 
By using the Hopf-Cole transformation of the principal eigenfunctions,
the problem is transformed into one about the limit behavior of effective Hamiltonian.
This is essentially a vanishing viscosity method.
In the almost periodic case,
however,
classical eigenfunctions do not exist in general.
We therefore introduce a class of functions
and use them to analyze the asymptotic behavior of the principal eigenvalues.

The asymptotic behavior when $L \to 0^+$ was investigated in much literature \cite{smaily2009homogenization,el2008pulsating,nadin2010effect}. 
For example, Smaily, Hamel and Roques \cite{smaily2009homogenization} established the homogenization phenomenon of the equation \eqref{equation} in symmetric divergence form with $\tilde{c} = 0$ . 
In one dimension, 
they proved that 
the spreading speed converges to $2\sqrt{<a>_H<c>}$ as $L \to 0^+$, 
where $<a>_H$ is the harmonic mean of $a$ and $<c>$ is the arithmetic mean of $c$,
respectively. 
They also investigated the asymptotic behavior of the spreading speeds in higher dimension for some special cases.
Subsequently,
Nadin \cite{nadin2010effect} employed a novel characterization of the periodic principal eigenvalue to derive an analogous result for arbitrary spatial dimensions.

The aforementioned results all rely on characterizing the spreading speeds in terms of the eigenvalues of the linearized operator near the unstable steady state $u \equiv 0$. 
More precisely, omitting the scale change, let $L=1$ and $\tilde{c} = 0$ and define a formal elliptic operator
\begin{equation}\label{LinOpe}
\Lcal \phi := \sum\limits_{i,j=1}^{n}a^{ij}D_{ij}\phi 
+ \sum\limits_{i=1}^{n}b^iD_i\phi 
+ c\phi \quad \hbox{ for all } \phi \in C^2(\RR^n).
\end{equation}
For any $p \in \RR^n$, define
\begin{equation}
\Lcal(p) \phi := e^{p \cdot x}\Lcal(e^{-p \cdot x}\phi) \quad \hbox{ for all } \phi \in C^2(\RR^n).
\end{equation}
Here, notice that when $p=0$, the zero vector, $\Lcal (0)$ has a identical form with $\Lcal$.
In the periodic case, where $a^{ij},b^i,c$ are assumed to be $\ZZ^n$-periodic, 
denote by $\lambda(p)$ the corresponding periodic principal eigenvalue of $\Lcal(p)$. 
That is
\begin{equation}\label{cellproblem}
\left\{\begin{array}{l}
\Lcal(p) \phi = \lambda(p)\phi, \\
\phi > 0, \ \phi \hbox{ is } \ZZ^n\hbox{-periodic}.
\end{array}\right.
\end{equation}
The existence and uniqueness of $\lambda(p)$ is promised by the Krein-Rutman theory.
Then under the condition $c > 0$, the spreading speed in direction $e$ is
\begin{equation}\label{speed}
\omega(e) = \inf_{p \cdot e > 0}\frac{\lambda(p)}{p \cdot e}, 
\end{equation}
see \cite{berestycki2008asymptotic,freidlin2020wavefront,gartner1979propagation,weinberger2002spreading}.

In almost periodic settings, 
as noted earlier, 
the operator $\Lcal(p)$ generally may not possess a positive eigenfunction within the class of almost periodic functions \cite{bjerklov2006positive,sorets1991positive}. 
To overcome this difficulty, Berestycki and Nadin \cite{berestycki2012spreading,berestycki2019asymptotic} introduce a new tool, {\bf the generalized principal eigenvalue of $\Lcal(p)$}, 
which is defined as following
\begin{equation}\label{lowerGPE}
\underline{\lambda_1}(p) := \sup \{\lambda : \exists \phi \in \Acal \hbox{ such that } \Lcal(p)\phi \geq \lambda \phi \hbox{ in } \RR^n\},
\end{equation}
\begin{equation}\label{upperGPE}
\overline{\lambda_1}(p) := \inf \{\lambda : \exists \phi \in \Acal \hbox{ such that } \Lcal(p)\phi \leq \lambda \phi \hbox{ in } \RR^n\},
\end{equation}
where
$$\Acal := \{\phi \in C^2(\RR^n) \cap W^{1,\infty}(\RR^n) :  \inf_{\RR^n} \phi > 0\}.$$
Using a result from Lions and Souganidis \cite{lions2005homogenization}, 
they proved the coincidence of these two quantities in the almost periodic case 
and thus we could define
\begin{equation}\label{GPE}
\lambda(p) := \underline{\lambda_1}(p) = \overline{\lambda_1}(p).
\end{equation}
Under the hypothesis
\begin{equation}\label{spreadingoccur}
\liminf\limits_{|x| \to +\infty}(4c(x)\min\limits_{e \in \mathbb{S}^{n-1}}(eA(x)e) - |B(x)|^2) > 0,
\end{equation}
they proved that the spreading speed of \eqref{equation} in direction $e$ is given by the same formula as in \eqref{speed}. 
The condition \eqref{spreadingoccur} was released to
\begin{equation}\label{positivereaction}
\inf_{x \in \RR}c(x) > 0.
\end{equation}
by Liang and Zhou \cite{liang2022propagation} in dimension $1$.

In conclusion, 
we are suggested to focus on the asymptotic behavior of the generalized principal eigenvalues as $L$ tends to infinity and zero respectively.
As will be shown,
The asymptotic behavior of the generalized principal eigenvalues in our problems leads to two distinct homogenization problem:
the $L \to \infty$ limit corresponds to the viscous Hamilton-Jacobi equations,
while the $L \to 0$ limit corresponds to the second-order case. 
Our analysis employs the doubling variable technique,
which provides a powerful framework for handling problems on unbounded domains.
Moreover, by deriving appropriate estimates, 
we are able to characterize the associated convergence rates.
While convergence rates have been studied for decades,
with some optimal results established in the periodic settings
\cite{armstrong2014error,qian2024optimal,tran2021hamilton,shen2015convergence,capuzzo2001rate},
the almost periodic case remains largely open.
To address this,
we introduce a quantitative measure of almost periodicity, 
following \cite{armstrong2014error,shen2015convergence},
which enables the derivation of convergence rates in this more general context.
As a corollary of our analysis,
we find that the generalized principal eigenvalues, $\lambda(p)$'s, are smooth in $p$,
which implies the redundancy of the conditions in \cite{nadin2017generalized}.

Beyond the results on the convergence as $L\to 0,+\infty$,  
this paper also aims to understand how the drift $B_L$ and the birth rate $\tilde{c}$ influence propagation at extreme scales.
Based on the analysis of the asymptotic behavior of effective Hamiltonians,  we show  a normal scale perturbation $\tilde{c}\not=0$ of the growth rate with mean zero,  strictly accelerates the propagation in the media with rapid or slow oscillation.
This directly stems from the homogenization process, where the rapidly and slowly varying variables decoupling, resulting in a situation where both variables exhibit heterogeneity relative to each other. 
Furthermore,
in the one-dimensional case with constant coefficients $a,c,\tilde{c}$,
we prove that a mean zero $b_L \not \equiv 0$ strictly decelerates propagation when $L$ is large enough.

This paper is organized as follows. 
In the next section, we recall the definition of the effective Hamiltonian and its connection to the generalized principal eigenvalues. 
We also review key properties of the effective Hamiltonian and state the main results of the paper. 
These include the asymptotic limits of the spreading speeds in arbitrary directions, the corresponding convergence rates, and several related findings. 
The proofs of these results are detailed in the subsequent sections.

\section{Preliminary and main results}

\subsection{The effective Hamiltonian}
In this subsection, 
we revisit the principal eigenvalue problem from the perspective of effective Hamiltonian theory,
by working exclusively with its Hopf-Cole transform.
More precisely,
define an elliptic operator
\begin{equation}\label{operator}
\begin{array}{l}\Lcal \phi:=
\sum\limits_{i,j=1}^{n}a_*^{ij}D_{ij}\phi 
+ \sum\limits_{i=1}^{n}b_*^i D_i\phi 
+ c_*\phi \quad  
\hbox{ in } \RR^n.\end{array}
\end{equation}

We say that $\Lcal$ is {\bf uniformly elliptic and almost periodic} provided that the coefficients of $\Lcal$ satisfies \thref{uniformelliptic} and \thref{almostperiodicity}.

We define the Hamiltonian associated with the operator \eqref{operator} as
\begin{equation}\label{Hamiltonian}
H : \RR^n \times \RR^n \times \Mcal^n \to \RR, 
(x,p,X) \mapsto \sum\limits_{i,j=1}^{n}a_*^{ij}(x)X_{ij} 
+ \sum\limits_{i,j=1}^{n}a_*^{ij}(x)p_ip_j 
+ \sum\limits_{i=1}^{n}b_*^i(x)p_i 
+ c_*(x),
\end{equation}  
where $\Mcal^n$ is the set of all $n \times n$ real matrix.
This formulation arises naturally from seeking a principal eigenpair $(\lambda,\phi)$ of \eqref{operator}.
Applying the Hopf-Cole transformation $\psi = \ln \phi$ leads to the equation
$H(x,D\psi,D^2\psi) = \lambda$.
Consequently, the problem of determining the principal eigenvalue is equivalent to solving this equation.

In the almost periodic case, we can still apply the Hamiltonian $H$ to consider the problem of the principal eigenvalue. 
However,
since a positive almost periodic eigenfunction may not exist for $\Lcal$,  
we instead seek a constant $\lambda$ such that for any $\delta > 0$, 
the approximation problem
\begin{equation}\label{deltacorrector}
\lambda-\delta \leq H(x,Dw,D^2w) \leq \lambda + \delta
\end{equation}
admits a bounded, Lipschitz continuous viscosity solution $w$. 
The constant $\lambda$ is so-called the effective Hamiltonian of the above problem:
\begin{definition}
Let $H$ be defined as \eqref{Hamiltonian}. 
If there exists $\lambda$ such that for any $\delta > 0$,
\eqref{deltacorrector} admits a bounded and Lipschitz continuous viscosity solution,
then we say the constant $\lambda$ is \textbf{the effective Hamiltonian} of the  problem
\begin{equation}\label{effectiveHamiltonian}
H(x,Dw,D^2w) = \lambda,
\end{equation}
and say $w$ is \textbf{a $\delta-$approximate corrector} of the problem \eqref{effectiveHamiltonian}.
\end{definition}
Lions and Souganidis \cite{lions2005homogenization} have proved the following proposition:
\begin{proposition}\cite{lions2005homogenization} 
Assume that $\Lcal$ is uniformly elliptic and almost periodic. 
Let $H$ be defined as \eqref{Hamiltonian}.  
Then the effective Hamiltonian $\lambda$ of $H(x,Dw,D^2w) = \lambda$ exists and is unique.
\end{proposition}

Following Lions and Souganidis \cite{lions2005homogenization},
for each $\epsilon > 0$ we consider the approximation problem
\begin{equation}\label{correc}
H(x,Du_\epsilon,D^2u_\epsilon) = \epsilon u_{\epsilon}.
\end{equation}
\begin{proposition}\thlabel{appro}\cite{lions2005homogenization,lions2005homogenization2}
For any $\epsilon > 0$, 
\eqref{correc} admits a unique solution $u_\epsilon$ that is bounded and Lipschitz continuous.

Moreover, one also has estimates
\begin{equation}\label{bound}
\|D^2u_{\epsilon}\|_{L^{\infty}} + \|Du_{\epsilon}\|_{L^{\infty}} + \|\epsilon u_{\epsilon}\|_{L^{\infty}} \leq C
\end{equation}
for some $C$ depending on $\|a_*^{ij}\|_{L^{\infty}},\|b_*^i\|_{L^{\infty}},\|c_*\|_{L^{\infty}}$  and $\alpha_m,\alpha_M$.

Finally, $\lim\limits_{\epsilon\to 0}\|\epsilon u_{\epsilon}-\lambda\|_{L^\infty}=0$, where $\lambda$ is the effective Hamiltonian of $H(x,Du,D^2u) =\lambda$.
\end{proposition}
This result will be used repeatedly in the sequel.
Back to the generalized principal eigenvalues of $\Lcal$,
we have

\begin{proposition}\thlabel{GPEandH}\cite{berestycki2019asymptotic}
Assume that $\Lcal$ is {uniformly elliptic and almost periodic }. Let $H$ be defined as \eqref{Hamiltonian}.
The generalized principal eigenvalue $\lambda$ of $\Lcal $ defined by \eqref{lowerGPE}-\eqref{GPE},
is well-defined and coincides with the effective Hamiltonian of \eqref{effectiveHamiltonian}.
\end{proposition}

Let us gather some properties of the generalized principal eigenvalues.

\begin{proposition}\cite{lions2005homogenization2}\thlabel{characterization of effective Ham}
Let $H$ be defined as \eqref{Hamiltonian}.
Assume $\lambda$ is the effective Hamiltonian of \eqref{effectiveHamiltonian}. Then for any $v$ that is bounded and Lipschitz continuous, we have
\begin{equation}
\inf_{x \in \RR^n}H(x,Dv,D^2v) \leq \lambda \leq \sup_{x \in \RR^n}H(x,Dv,D^2v)
\end{equation}
in viscosity sense.
\end{proposition}

Now for any $p\in \RR^n$,
define the operator:
\begin{equation*}
\begin{aligned}
\Lcal(p) \phi 
&:= e^{p\cdot x}\Lcal(e^{-p\cdot x}\phi) \\
&= \sum\limits_{i,j=1}^{n}a^{ij}_*D_{ij}\phi
+ \sum\limits_{i=1}^{n}(b_*^i - 2\sum_{j=1}^{n}a^{ij}_*p_j)D_i\phi
+ (\sum\limits_{i,j=1}^{n}a_*^{ij}p_ip_j - \sum\limits_{i=1}^{n}b^i_*p_i + c_*)\phi.
\end{aligned}
\end{equation*}
it is straightforward to verify that $\Lcal(p)$ remains uniformly elliptic and almost periodic.
Consequently,
by the preceding theory,
it admits a generalized principal eigenvalue,
which we denote by $\lambda(p)$.
The following proposition presents the relation between $\lambda(p)$ and $H$.

\begin{proposition}\thlabel{Lp}
For any $p\in \RR^n$,
$\lambda(p)$ is the effective Hamiltonian of $H(x,Dw-p,D^2w) = \lambda(p)$.
\end{proposition}

\begin{proof}
One has
\begin{equation*}
\begin{aligned}
H(x,q-p,X) 
&= \sum_{i,j=1}^{n}a_*^{ij}X_{ij} + \sum_{i,j=1}^{n}a_*^{ij}(q_i - p_i)(q_j - p_j) + \sum_{i=1}^{n}b^i_*(q_i - p_i) + c_* \\
&=  \sum_{i,j=1}^{n}a_*^{ij}X_{ij} 
+ \sum_{i,j=1}^{n}a_*^{ij}q_iq_j \\
&\quad + \sum_{i=1}^{n}(b^i_* - 2\sum_{j=1}^{n}a^{ij}_*p_j)q_i 
+ \sum_{i,j=1}^{n}a^{ij}_*p_ip_j - \sum_{i=1}^{n}b^i_*p_i + c_*, \\
\end{aligned}
\end{equation*}
which is the corresponding Hamiltonian of $\Lcal(p)$.
\end{proof}

%Based on the definition of the generalized principal eigenvalue, \cite{berestycki2019asymptotic}  obtained the following result on the convexity of $\lambda(p)$:
\begin{proposition}\thlabel{convexity}
$p \mapsto \lambda(p)$ is convex and smooth.
\end{proposition}
The convexity is from \cite{berestycki2019asymptotic}. 
The smoothness is proved in \thref{smoothness of the effective Hamiltonian} in this paper.

To conclude this subsection, 
we introduce the effective Hamiltonian associated with a first-order problem.

\begin{assumption}\thlabel{Ishii}  
Suppose $\overline{H} \in C(\RR^n \times \RR^n)$ satisfies the following hypotheses:

\noindent(1) $\lim\limits_{R \to +\infty}\inf\limits_{x \in \RR^n}\{\overline{H}(x,p): |p| \geq R\} = \infty$. 

\noindent(2) $\overline{H}$ is locally Lipschitz continuous in $p$ uniformly in $x$.

\noindent(3) $\overline{H}$ is bounded in $\RR^n \times B(0,R)$ for all $R > 0$.

\noindent(4) $\overline{H}$ is convex in $p$.

\noindent(5) $\overline{H}$ is almost periodic in $x$ locally uniformly in $p$.
\end{assumption}
Then Ishii \cite{ishii2000almost} has proved 
\begin{proposition} \cite{ishii2000almost}
Assume $\overline{H}$ satisfies \thref{Ishii}.	
Then for any $p \in \RR^n$,
one could also find a unique $\lambda(p)$ such that for any $\delta > 0$,
the problem
\begin{equation}\label{effecitveHamiltonianIshii}
\lambda(p) - \delta \leq \overline{H}(x,Dw+p) \leq \lambda(p) + \delta
\end{equation} 
admits a solution $w$ that is bounded and uniformly continuous.

Furthermore, for any $v$ that is bounded and Lipschitz continuous, we have
\begin{equation}
\inf_{x \in \RR^n}\overline{H}(x,Dv+p) \leq \lambda(p) \leq \sup_{x \in \RR^n}\overline{H}(x,Dv+p)
\end{equation}
in viscosity sense. 

Finally, $\lambda(p)$ is convex in $p$. 

{We call $\lambda(p)$ the effective Hamiltonian of $\overline{H}(x,Dv+p) =\lambda(p)$.  }

\end{proposition}

\subsection{Main results}

In the rest part of this paper, we assume that 
$$\inf_{x \in \RR^n}c(x) + \inf_{x \in \RR^n}\tilde{c}(x) > 0,$$
and
\begin{equation}\label{spreading occurs}
\sup_{x \in \RR^n}|B(x)|^2 < 4 \alpha_m( \inf_{x \in \RR^n}c(x) + \inf_{x \in \RR^n}\tilde{c}(x))
\end{equation} 
as in \cite{berestycki2012spreading,berestycki2019asymptotic}, 
where condition \eqref{spreading occurs} promises the spreading will occur in every direction.
For any $e \in \mathbb{S}^{n-1}$,
the spreading speed along the direction $e$,
$\omega(e)$, defined by
\begin{equation*}
\begin{cases}
\liminf\limits_{t \to +\infty}u(t,x + \omega t e) = 1 \quad 
\hbox{ if } 0 \leq \omega < \omega^*(e), \\
\lim\limits_{t \to +\infty}u(t,x + \omega t e) = 0 \quad 
\hbox{ if } \omega > \omega^*(e).
\end{cases}
\end{equation*}
exists thanks to \cite{berestycki2019asymptotic}.
Their work also established that these spreading speeds are given by the generalized principal eigenvalues of the linearized operator:
For any $L > 0$,
define
\begin{equation*}
\Lcal(L) \phi 
:= \sum_{i,j=1}^{n}a^{ij}_LD_{ij}\phi 
+ \sum_{i=1}^{n}b^i_L D_i\phi 
+ (c_L + \tilde{c})\phi.
\end{equation*}
For any $p \in \RR^n$,
define
\begin{equation*}
\Lcal(L,p) \phi := e^{p\cdot x}\Lcal(L)(e^{-p\cdot x}\phi),
\end{equation*}
and denote $\lambda(L,p)$ the corresponding generalized principal eigenvalue.
Then we have
\begin{theorem}\cite{berestycki2019asymptotic}
For any $e \in \mathbb{S}^{n-1}$ and $L > 0$, 
the spreading speed of \eqref{equation} in direction $e$, 
write $\omega(e;L)$, 
always exists. 
Moreover,
$$\omega(e;L) = \inf\limits_{p \cdot e > 0}\frac{\lambda(L,p)}{p \cdot e}.$$
\end{theorem}

Here
we would like to introduce a two-scale Hamiltonian:
\begin{equation}\label{2-scaleH}
\tilde H(x,y,p,X) = \sum\limits_{i,j=1}^{n}a^{ij}(y)X_{ij} 
+ \sum\limits_{i,j=1}^{n}a^{ij}(y)p_ip_j 
+ \sum\limits_{i=1}^{n}b^i(y)p_i 
+ c(y) + \tilde{c}(x).
\end{equation}
Then by a direction computation and \thref{Lp},
we see $\lambda(L,p)$ is the effective Hamiltonian of 
\begin{equation}
\tilde H(x,\frac{x}{L},Du-p,D^2u) = \lambda.
\end{equation}

Throughout the remainder of this paper, 
the constant $C$ may vary from line to line.
We now present the results in the small-scale limit,
beginning with the asymptotic limits of the generalized principal eigenvalues.

\begin{theorem}\thlabel{maintheoremzero}
For any $p \in \RR^n$, the limit
$$\overline\lambda(p) = \lim\limits_{L \to 0^+}\lambda(L,p)$$
exists and it is the effective Hamiltonian of the following problem
\begin{equation}\label{zeroeigenvalues}
\begin{aligned}
\sum\limits_{i,j=1}^{n}\iota(a^{ij})D_{ij}u 
+ \sum\limits_{i,j=1}^{n}\iota(a^{ij})(D_iu-p_i)(D_ju-p_j)
+ \sum\limits_{i=1}^{n}\iota(b^i)(D_iu-p_i) 
+ \iota(c) + \tilde{c} = \overline\lambda(p),
\end{aligned}
\end{equation}
where $\iota$ is a functional depending only on $A$ maps the set of almost periodic functions to $\RR$.
Furthermore, 
the convergence $\lambda(L,p) \to \overline \lambda(p)$ as $L \to 0^+$ is locally uniformly with respect to $p \in \RR^n$.
\end{theorem}

\begin{remark}
The functional $\iota$ will be explicitly explained by \thref{iota}.
\end{remark}

\begin{remark}
As established earlier, from the perspective of principal eigenvalue problems, the generalized principal eigenvalue $\lambda(L,p)$ converges formally to that of the following problem:
\begin{equation*}
\begin{cases}
\overline{\Lcal}(p)\phi 
:= e^{p \cdot x} \overline{\Lcal}(e^{-p \cdot x}\phi)
=\overline \lambda(p) \phi, \\
\phi > 0, \quad 
\phi \hbox{ is bounded},
\end{cases}
\end{equation*}
where
\begin{equation*}
\overline{\Lcal} \phi :=
\sum\limits_{i,j=1}^{n}\iota(a^{ij})D_{ij}\phi 
+ \sum\limits_{i=1}^{n}\iota(b^i)D_i\phi
+ \iota(c) + \tilde{c}. 
\end{equation*}
\end{remark}

For the spreading speeds,
we have

\begin{theorem}\thlabel{zero limit}
$\omega(e;L)$ converges as $L \to 0$ for any $e \in \mathbb{S}^{n-1}$.
Moreover,
the limit is characterized by
\begin{equation*}
\lim\limits_{L \to 0} \omega(e;L) = \inf_{p \cdot e > 0} \frac{\overline\lambda(p)}{p \cdot e}.
\end{equation*}
\end{theorem}

As a corollary,
we have
\begin{corollary}
The limits of the spreading speeds of \eqref{equation} as $L \to 0$ are coincident with the speeds of the following {\bf homogenized Fisher-KPP equation} of \eqref{equation}
\begin{equation*}
\begin{cases}
u_t - \sum\limits_{i,j=1}^{n}\iota(a^{ij})D_{ij}u 
- \sum\limits_{i=1}^{n}\iota(b^i)D_iu = (\iota(c) + \tilde{c}(x))u(1-u), \\
u(0,x) = u_0(x),
\end{cases}
\end{equation*}
where $u_0 \in C_c(\RR^n, [0,1])$ and $u_0 \not \equiv 0$. 
\end{corollary}

In the limit as $L \to 0$ and for $\tilde{c} = 0$,
we have:
\begin{corollary}\thlabel{zeroexpression}
If $\tilde{c} = 0$, then
$$\overline\lambda(p) = \sum\limits_{i,j=1}^{n}\iota(a^{ij})p_ip_j 
- \sum\limits_{i=1}^{n}\iota(b^i)p_i 
+ \iota(c).$$
\end{corollary}
In one dimension,
we obtain an explicit characterization of both $\iota$ and the homogenized Fisher-KPP equation:
\begin{proposition}\thlabel{iota expression}
When $n=1$,
we have
\begin{equation}\label{iotaexpression}
\iota(f) =<a>_H<a^{-1}f> 
\end{equation}
where $<f>_H:=<f^{-1}>^{-1}$ and 
\begin{equation*}
<f> := \Xint{-}f(x)dx = \lim\limits_{R \to +\infty} \frac{1}{|B(0,R)|}\int_{B(0,R)}f(x)dx,
\end{equation*}
the latter is well-defined for any almost periodic functions by \cite{bochner1927beitrage}.

Therefore, the  {\bf homogenized Fisher-KPP equation} of  \eqref{equation}
is 
\begin{equation}
u_t=<a>_Hu_{xx}+<a>_H<a^{-1}b> u_x+( <a>_H<a^{-1}c> + \tilde{c} )u(1-u). 
\end{equation}
\end{proposition}

\begin{remark}
A key distinction arises between our non-divergence form model and the divergence form case studied in \cite{smaily2009homogenization}.
Specifically, \cite{smaily2009homogenization} considers the equation $u_t = (a(x/L)u_x)_x + c(x/L)u(1-u)$ with smooth, positive, 1-periodic coefficients,
which homogenizes to
\begin{equation*}
u_t = <a>_H u_{xx} + <c> u(1-u).
\end{equation*}
The difference stems from the presence of the large advection term $\frac{1}{L}a'(x/L)u_x$ in the divergence form model as $L \to 0$.
This term modifies the corrector equation and,
consequently,
the homogenized limit. 
\end{remark}

To quantify the convergence rates of the spreading speeds, 
we introduce the following measure of almost periodicity, 
inspired by the work of Armstrong, Cardaliaguet, Souganidis, and Shen \cite{armstrong2014error,shen2015convergence}:
\begin{equation}\label{quantityofa.p.}
\rho(R;f) := \sup_{y \in \RR^n}\inf_{z \in \RR^n, |z| \leq R}\|f(\cdot + y) - f(\cdot + z)\|_{L^{\infty}(\RR^n)}
\end{equation}
for any $R > 0$ and vector-valued function $f$. 
One can show that a bounded continuous function $f$ in $\RR^n$ is almost periodic if and only if $\rho(R;f) \to 0$ as $R \to +\infty$. 
Define
\begin{equation*}
\Theta_\sigma(r;f) = \inf_{0 \leq R < 1/r}(\rho(R;f) + (rR)^\sigma).
\end{equation*}
Then the continuity and monotonicity of $\rho(R;f)$ imply
$$\Theta_\sigma(\cdot;f): \RR^+ \to \RR^+ \hbox{ is continuous, increasing and } \Theta_\sigma(0^+;f) = 0.$$
In other words, $\Theta_\sigma(\cdot;f)$ is a modulus.

Using the characterization of almost periodicity above, 
we can obtain the following estimates.

\begin{theorem}\thlabel{zerorate}
For any $\sigma \in (0,1)$,
there exists $C_\sigma > 0$ such that
$$\|\omega(e;L) - \omega(e;0)\| \leq C_\sigma\Theta_\sigma(L;(A,B,c)).$$
Moreover,
if $A,B$ and $c$ are periodic with the same period, then we further have
$$\|\omega(e;L) - \omega(e;0)\| \leq C_\sigma L^\sigma.$$
\end{theorem}

We now turn to the case $L \to \infty$.

\begin{theorem}\thlabel{maintheoreminfinity}
For any $p \in \RR^n$, the limit
$$\Lambda(p) = \lim\limits_{L \to +\infty}\lambda(L,p)$$
exists and coincides with the effective Hamiltonian of the problem
\begin{equation}
\overline{H}(x,Du-p) = \Lambda(p).
\end{equation}
Here $\overline{H}(x,p)$ itself is the effective Hamiltonian of the following  problem
\begin{equation}\label{transHamilton}
\tilde H(y,x,D\psi + p, D^2\psi) = \overline{H}(x,p),
\end{equation}
with differentiation in $y$,
where $\tilde H$ is defined in \eqref{2-scaleH}. 
Moreover, 
the convergence $\lambda(L,p) \to \Lambda(p)$ as $L \to +\infty$ is locally uniformly with respect to $p \in \RR^n$.
\end{theorem}

\begin{theorem}\thlabel{infity limit}
$\omega(e;L)$ converges as $L \to +\infty$ for any $e \in \mathbb{S}^{n-1}$,
and the limit is characterized by
$$\lim\limits_{L \to +\infty}\omega(e;L) = \inf_{p \cdot e> 0}\frac{\Lambda(p)}{p \cdot e}.$$
\end{theorem}

Regarding the convergence rate,
we need to impose some restriction on $\tilde{c}$.

\begin{theorem}\thlabel{infinityrate}
Suppose there exist constants $C, \tau > 0$ such that
\begin{equation}\label{a.p.char}
\rho(R;\tilde{c}) \leq CR^{-\tau}
\end{equation}
for all $R \geq 1$.
Then there exists a constant $C > 0$ for which
$$\|\omega(e;L) - \omega(e;+\infty)\| \leq C L^{\frac{-\tau}{2\tau+1}}.$$
Furthermore,
if $\tilde{c}$ is periodic,
the rate improves to
$$\|\omega(e;L) - \omega(e;+\infty)\| \leq CL^{-\frac{1}{2}}.$$
\end{theorem}

We present several additional results.
First,
compared to \cite{hamel2011viscosity},
\thref{maintheoreminfinity} can be viewed as an extension of their work.
In fact,
in the case $\tilde{c} = 0$,
it is straightforward to verify that
\begin{equation*}
\overline{H}(x,p) 
= \sum\limits_{i,j=1}^{n}a^{ij}(x)p_ip_j 
+ \sum\limits_{i=1}^{n}b^i(x) p_i
+ c(x),
\end{equation*}
which recovers the expression given in \cite{hamel2011viscosity}. 
Moreover,
we have the following almost periodic version:

\begin{proposition}\thlabel{infinity in 1-dim}
If  $n = 1$ and $\tilde{c} = 0$, then 
$$\Lambda(p) = \begin{cases}
(j_+)^{-1}(p), \quad &p > j_+(M), \\
M, \quad &j_-(M) \leq p \leq j_+(M), \\
(j_-)^{-1}(p), \quad &p < j_-(M),
\end{cases}$$
where
\begin{equation}\label{infirelation}
j_{\pm}: [M,+\infty) \to \RR, \quad \lambda \mapsto \Xint-\frac{b(x)}{2a(x)}dx \pm \Xint-\sqrt{\frac{\lambda - c(x)}{a(x)} + \frac{b^2(x)}{4a^2(x)}}dx,
\end{equation}
and
$$M := \sup\limits_{x \in \RR}\{c(x) - \frac{b^2}{4a}(x)\}.$$
\end{proposition}

In particular,
we have
\begin{proposition}\thlabel{smaller eigenvalues}
Assume $n = 1$, $\tilde{c} = 0$ and $a, c$ are positive constants,
$b \not \equiv 0$ is almost periodic with $<b> = 0$.
Then we have
\begin{equation*}
\Lambda(p) \leq ap^2 + c.
\end{equation*}
The equality $\Lambda(p) = ap^2 + c$ holds if and only if $p = 0$.
\end{proposition}

As an application,
we have
\begin{corollary}\thlabel{slow-down speeds}
Assume  $n = 1$, $\tilde{c} = 0$ and $a, c$ are positive constants,
$b \not \equiv 0$ is almost periodic with $<b> = 0$.
Then
$\omega(\pm 1; L) < 2\sqrt{ac}$ when $L$ is sufficiently large.
\end{corollary}
\begin{remark}
In periodic settings,
let us recall the related conclusions in \cite{berestycki2005speed} and \cite{nadin2011some}. Assume $A=I$, the identical matrix, and constant $c,\tilde{c} > 0$.
it is shown in \cite{berestycki2005speed} that an incompressible, mean zero drift $B\not\equiv 0$ accelerate the propagation.
In contrast,
\cite{nadin2011some} demonstrated that a gradient drift
$B = \nabla Q$, for some periodic function $Q$,
decelerates propagation. 
Moreover,
The proof in \cite{nadin2011some} implies that \thref{slow-down speeds} holds for any $L > 0$ when $b$ is periodic.

Here, we leave a conjecture here:
\begin{conjecture}Suppose $B = \nabla Q$ for some almost periodic function $Q$ and $B\not \equiv 0$.
Then $B$ will strictly slow-down the propagation when $A=I$, the identical matrix, $c$ is a positive constant for any $L > 0$.
\end{conjecture}
\end{remark}

We conclude this section with the following result.
\begin{theorem}\thlabel{accelerate}
Let $ \tilde{c} = 0$,
and denote the spreading speeds of \eqref{equation} along $e \in \mathbb{S}^{n-1}$ by $\omega(e;L)$.
Then for any $\tilde{c} \not \equiv 0$ with $<\tilde{c}> = 0$,
the corresponding spreading speeds of \eqref{equation} along $e \in \mathbb{S}^{n-1}$ are strictly greater than $\omega(e;L)$ for sufficiently large or small $L$.
\end{theorem}

\begin{remark}
We note that Berestycki, Hamel and Roques \cite{berestycki2005analysis} have proved that when $A$ is the identical matrix,
$B = 0$, ${c}$ is a nonnegative constant and $\tilde{c} = R v(x)$,
where $v$ is periodic, $\Xint{-}v \geq 0 \, v\not\equiv 0$, $R\geq 0$ is constant.
Then the spreading speed is an increasing function of $R $.
In particular,
if $\Xint{-}v = 0$ but $v \not\equiv 0$,
then their work implies that introducing a heterogeneous reaction term always accelerate the propagation.

Here we conclude that the above Theorem \thref{accelerate} does not hold for general $L$ by providing a simple counterexample in one dimension:
Let $a = 1$,
$b = 0$,
$c = \sin{x} + 2$
and $\tilde{c} = -\sin{x}$.
Then for $L = 1$,
the spreading speeds of \eqref{equation} is strictly smaller than the speed when $\tilde{c} \equiv 0$. 
\end{remark}

\section{Limit behavior when $L \to 0$ : Proof of \thref{maintheoremzero,zero limit,zerorate}}

In this section, 
we mainly consider the limit behavior of the generalized principal eigenvalue as $L \to 0$. 
In the rest of this paper, 
we will sometimes drop the subscript of $\|\cdot\|_{L^{\infty}(\RR^n)}$ when no confusion arises.

Based on \thref{appro}, for any $L > 0$ and $p \in \RR^n$,
we approximate the generalized principal eigenvalue $\lambda(L,p)$ of $\Lcal(L,p)$
by considering the following problem  
\begin{equation}\label{approximation}
\tilde H(x,\frac{x}{L},Du_L-p,D^2u_L) = \epsilon u_L,
\end{equation}
where $\tilde H$ is defined in \eqref{2-scaleH}.

For this purpose, we introduce the doubling variable technique.  Let $v_0=v_0(x)$ and $v_2=v_2(x,y), $$$v_L(x) = v_0(x) + L^2 v_2(x,\frac{x}{L}),$$
use $v_L$ as a test function of \eqref{approximation} and calculate by composite derivative directly: 
\begin{equation*}
\begin{aligned}
\epsilon v_L - \tilde H(x,\frac{x}{L},Dv_L-p,D^2v_L)
&= \Big(-\sum\limits_{i,j=1}^{n}a^{ij}(y)D_{y_iy_j}v_2 
- \sum\limits_{i,j=1}^{n}a^{ij}(y)D_{ij}v_0 \\
&\quad -\sum\limits_{i,j=1}^{n}a^{ij}(y)(D_iv_0-p_i)(D_jv_0-p_j)
- \sum\limits_{i=1}^{n}b^i(y)(D_iv_0-p_i)  \\
&\quad -c(y) - \tilde{c}(x)+ \epsilon v_0  \Big)\Big|_{y=\frac{x}{L}}+ \rm o(1)\\
&=: I|_{y=\frac{x}{L}}+{\rm o(1)}, (L\to 0).
\end{aligned}
\end{equation*}
Here $\rm o(1)$ is about $L\to 0$, and depends on $v_0,v_2$ and $x$.

For $v_L$ to be a good approximation of the unique bounded solution $u_L$ of \eqref{approximation},
a natural approach is to require the following general equation with two variables $(x,y)$  satisfied: $I\equiv 0$, that is,
\begin{equation}\label{dvequation}
\begin{aligned}
&\sum\limits_{i,j=1}^{n}a^{ij}(y)D_{y_iy_j}v_2 
+ \sum\limits_{i,j=1}^{n} a^{ij}(y)D_{ij}v_0 
+ \sum\limits_{i,j=1}^{n}a^{ij}(y)(D_iv_0-p_i)(D_jv_0 - p_j) \\
&+ \sum\limits_{i=1}^{n}b^i(y)(D_iv_0-p_i) 
+ c(y) = \epsilon v_0 - \tilde{c}(x),
\end{aligned}
\end{equation}

Since the right-hand side of \eqref{dvequation} depends only on $x$,
it is natural to choose $v_2$ so that the left-hand side is also independent of $y$.
We therefore turn to study
\begin{equation}\label{homogenizationoperator}
\sum\limits_{i,j=1}^{n}a^{ij}(y)D_{ij}v(y) + F(y) = \iota(F),
\end{equation}
where $F$ is almost periodic and $\iota(F) \in \RR$.
%We will prove that for any almost periodic $F$, one could define $\iota(F)$ the effective Hamiltonian of the above problem.
We will address equation \eqref{homogenizationoperator} in a manner similar to \eqref{effectiveHamiltonian}.

\begin{theorem}\thlabel{iota}
Assume $F$ is almost periodic.
Then there exists a unique constant $\iota(F) \in \RR$ such that
for any $\delta > 0$,
the problem
\begin{equation*}
\iota(F) - \delta < \sum\limits_{i,j=1}^{n}a^{ij}(y)D_{ij}v(y) + F(y) < \iota(F) + \delta
\end{equation*}
admits a bounded and Lipschitz continuous viscosity solution $v$.
\end{theorem}

\begin{proof}
For any $\epsilon > 0$, 
by classical elliptic theory, 
there exists $v_{\epsilon} \in W^{2,p}_{\mathrm{loc}}(\RR^n)$ for any $p \in (1,\infty)$ satisfying
$$\sum\limits_{i,j=1}^{n}a^{ij}(y)D_{ij}v_\epsilon + F(y) = \epsilon v_{\epsilon},$$
with estimates
$$\|\sqrt{\epsilon}Dv_\epsilon\| + \|\epsilon v_{\epsilon}\| \leq C\|F\|.$$
and
$$\Xint{-}_{B(x,R^n)}|D^2v_\epsilon(x)|^pdx \leq C\|F\|$$
for any $x \in \RR^n$ and $R > 0$.

Now for any $z\in \RR^n$, let $\varphi(y) = v_{\epsilon}(y) - v_\epsilon(y+z)$. Then $\varphi$ solves
$$\sum\limits_{i,j=1}^{n}a^{ij}(y)D_{ij}\varphi -\epsilon \varphi = \hbox{tr}((A(y+z) - A(y))D^2v_\epsilon(y+z)) + F(y+z) - F(y).$$
By classical elliptic regularity theory, we have
$$\epsilon\|v_\epsilon(\cdot) - v_\epsilon(\cdot+z)\| \leq C(\|A(\cdot) - A(\cdot+z)\| + \|F(\cdot) - F(\cdot + z)\|).$$
Note that $A$ and $F$ are almost periodic, 
thus for any $R$, there exists $\tilde{z} \in \RR^n$ with $|\tilde{z}| \leq R$ such that
$$\|A(\cdot+z) - A(\cdot+\tilde{z})\| + \|F(\cdot+z) - F(\cdot+\tilde{z})\| \leq \rho(R;(A,F))$$
for all $z \in \RR^n$.
Using this estimate then we obtain
\begin{equation*}
\begin{aligned}
\epsilon\|v_\epsilon(\cdot) - v_\epsilon(\cdot+z)\| 
&\leq \epsilon\|v_\epsilon(\cdot) - v_\epsilon(\cdot+\tilde{z})\| + \epsilon\|v_\epsilon(\cdot+\tilde{z}) - v_\epsilon(\cdot+z)\| \\
&\leq \epsilon\|Dv_\epsilon\|\|\tilde{z}\| + C\rho(R;(A,F)) \\
&\leq C\sqrt{\epsilon}R + C\rho(R;(A,F)).
\end{aligned}
\end{equation*}
Then we have
$$\epsilon\|v_{\epsilon} - v_\epsilon(\cdot+z)\| \leq C\Theta_1(\sqrt{\epsilon};(A,F)) \to 0 \quad \hbox{ as } \epsilon \to 0^+,$$
which implies that $\epsilon v_\epsilon$ converges to some constant.

The uniqueness of $\iota(F)$ follows from the same argument as in \cite{lions2005homogenization2}.
\end{proof}

By the argument in \cite{lions2005homogenization2},
we also have the following:

\begin{proposition}\thlabel{characterization of iota}
For any bounded and Lipschitz continuous $v$,
we have
\begin{equation*}
\inf_{y \in \RR^n} \sum\limits_{i,j=1}^{n}a^{ij}(y)D_{ij}v(y) + F(y) 
\leq \iota(F) 
\leq \sup_{y \in \RR^n} \sum\limits_{i,j=1}^{n}a^{ij}(y)D_{ij}v(y) + F(y)
\end{equation*}
in viscosity sense.
\end{proposition}

One could easily verify that $\iota$ defines a linear map from $AP(\RR^n)$ to $\RR$.
Here $AP(\RR^n)$ denotes the space of all almost periodic functions on $\RR^n$ equipped with the supremum norm. 
This space was shown by Bohr \cite{bohr1925theorie} to be complete.
Moreover, we see that
$\iota(F) \leq \|F\|.$ This means that $\iota$ is a bounded linear functional in $AP(\RR^n)$ with $L^\infty$ topology.

The definition of the operator $\iota$ implies that $v_0$ must satisfy the equation
\begin{equation}\label{limitat0}
 \sum\limits_{i,j=1}^{n}\iota(a^{ij})D_{ij}v_0 
+ \sum\limits_{i,j=1}^{n}\iota(a^{ij})(D_iv_0-p_i)(D_jv_0 - p_j) 
+ \sum\limits_{j=1}^{n} \iota(b^i)(D_iv_0-p_i) 
+ \iota(c) 
+ \tilde{c}(x) 
= \epsilon v_0.
\end{equation}

For any direction $\xi \in \mathbb{S}^{n-1}$,
\thref{uniformelliptic} gives
$$\inf_{x \in \RR^n}  \sum\limits_{i,j=1}^{n}a^{ij}(x) \xi_i \xi_j \geq \alpha_m.$$
Taking $F = \sum\limits_{i,j=1}^{n}a^{ij} \xi_i\xi_j$ and applying \thref{characterization of iota}, 
we obtain
$$\sum\limits_{i,j=1}^{n}\iota(a^{ij}) \xi_i\xi_j = \sum\limits_{i,j=1}^{n}\iota(a^{ij} \xi_i\xi_j)  \geq \alpha_m,$$
which implies that $(\iota(a^{ij}))_{n\times n}$ is positive definite.
By Perron's method, \eqref{limitat0} admits a unique bounded solution $v_0$.

For any $p \in \RR^n$,
$L > 0$ and $\epsilon > 0$,
let $u_L$ be the unique bounded solution of \eqref{approximation}.
The strategy of proving \thref{maintheoremzero} is to show that $u_L$ converges to $v_0$ as $L \to 0^+$, 
and therefore $\lambda(L,p)$ converges to certain effective Hamiltonian.

\begin{proof}
[\textbf{Proof of \thref{maintheoremzero}}]
Fix $\epsilon > 0$ and $p \in \RR^n$. 
For any $L > 0$, 
let $u_{L}$ be the unique bounded solution of \eqref{approximation}.
The estimates \eqref{bound} implies that
 $\{u_L\}_{L > 0}$ is bounded in $(C^2(\RR^n),\|\cdot\|_{W^{2,\infty}(\RR^n)})$. 
By Arzela-Ascoli theorem, there exists a subsequence of $\{u_L\}_{L > 0}$ that converges in to a limit $v \in C^1(\RR^n)$ as $L \to 0^+$. 

For any $x_0 \in \RR^n$, denote $q = Dv(x_0)$.
Assume $\phi \in C^2(\RR^n)$ touches $v$ at $x_0$ from below.
For any $\delta > 0$, let $\psi = \psi(y)$ a $\delta$-corrector of the following  problem
\begin{equation*}
\|\sum\limits_{i,j=1}^{n}a^{ij}(y)D_{ij}\psi 
+ \sum\limits_{i,j=1}^{n}a^{ij}(y)D_{ij}\phi(x_0) 
+ \sum\limits_{i,j=1}^{n}a^{ij}(y)(q_i-p_i)(q_j-p_j) 
+ \sum\limits_{i=1}^{n}b^i(y)(q_i-p_i) + c(y) - N\| \leq \delta,
\end{equation*}
where 
\begin{equation*}
N = \iota(\sum\limits_{i,j=1}^{n}a^{ij}D_{ij}\phi(x_0) 
+ \sum\limits_{i,j=1}^{n}a^{ij}(q_i-p_i)(q_j-p_j) 
+ \sum\limits_{i=1}^{n}b^i(q_i-p_i) + c) \in \RR.
\end{equation*}
Let $\psi_L(x) = u_L(x) - \phi(x) - L^2\psi(\frac{x}{L})$. 
Note that $\psi_L \to v-\phi$ as $L \to 0$ locally uniformly and $v-\phi$ attains its minimum at $x = x_0$, 
then there exists $\{x_L\}_{L > 0}$ with $x_L \to x_0$ as $L \to 0$ such that $\psi_L$ has a local minimum at $x = x_L$.
Therefore we have
\begin{equation*}
Du_L(x_L) = \phi(x_L) + L\psi(x_L/L), \quad D^2u_L(x_L) \geq D^2\phi(x_L) + D^2\psi(x_L/L).
\end{equation*}
Substituting the above inequality into the equation for $u_L$ at $x = x_L$ yields
\begin{equation*}
\begin{aligned}
\epsilon u_L 
&\geq \tilde{c} + c_L 
+ \sum\limits_{i=1}^{n}b_L^i(D_i\phi + LD_i\psi(\cdot/L) - p_i) 
+ \sum\limits_{i,j=1}^{n}(a_L^{ij}D_{ij}\phi + a_L^{ij}D_{ij}\psi(\cdot/L)) \\
&\quad + \sum\limits_{i,j=1}^{n}a_L^{ij}(D_i\phi + LD_i\psi(\cdot/L) - p_i)(D_j\phi + LD_j\psi(\cdot/L) - p_j). 
\end{aligned}
\end{equation*}
Using the inequality $\psi$ satisfies, we therefore get that
\begin{equation*}
\begin{aligned}
\epsilon u_L
&\geq \tilde{c} + N - \delta 
+ \sum\limits_{i=1}^{n}b_L^i(D_i\phi - q_i) 
+ \sum\limits_{i,j=1}^{n}a_L^{ij}(D_i\phi - q_i)(D_j\phi - q_j) \\
&+ \sum\limits_{i,j=1}^{n}a_L^{ij}(D_{ij}\phi - D_{ij}\phi(x_0))+ \rm o(1),
\end{aligned}
\end{equation*}
where the above inequalities take values at $x = x_L$.
Letting $L \to 0$ we have
$$\epsilon v(x_0) \geq \tilde{c}(x_0) + N - \delta.$$
Letting $\delta \to 0$, we conclude that $v$ is a viscosity supersolution of \eqref{limitat0}.
A parallel argument shows that $v$ is also a subsolution,
and hence a viscosity solution of \eqref{limitat0}.

It remains to show that $\epsilon u_L$ converges to $\lambda(L,p)$ as $\epsilon \to 0^+$ uniformly in $L$.
In fact, for each $\gamma > 0$, 
by the almost periodicity of $\tilde{c}$,
there exists a relatively dense set 
$\{\tilde{y}_n\}_{n \in \ZZ} \subset \RR^n$
such that
$$
\|\tilde{c}(\cdot) - \tilde{c}(\cdot + \tilde{y}_n)\| \leq \gamma/3
$$
for all $n \in \ZZ$.
We assume that $M_1 > 0$ satisfies $\mathop{\bigcup}\limits_{n \in \ZZ}B(\tilde{y}_n,M_1) = \RR^n$.
On the other hand, 
for $a^{ij},b^i$ and $c$, 
there exists $\{y_n\}_{n \in \ZZ} \subset \RR^n$ 
satisfying $\mathop{\bigcup}\limits_{n \in \ZZ}B(y_n,M_2) = \RR^n$ 
for some $M_2 > 0$ such that
$$\|A(\cdot)-A(\cdot+y_n)\| 
+ \|B(\cdot)-B(\cdot+y_n)\| 
+ \|c(\cdot)-c(\cdot+y_n)\| 
\leq \gamma/3$$
for all $n \in \ZZ$.
Since $\tilde{c}$ is uniformly continuous, 
there exists $\eta_0 > 0$ such that for all $\eta \in (0,\eta_0)$, 
$$\|\tilde{c}(\cdot + \eta) - \tilde{c}(\cdot)\| \leq \gamma/3.$$
By the relative density of the sets $\{y_n\}_{n \in \ZZ}$,
there exists $L_0 > 0$ such that for all $L < L_0$ and $n \in \ZZ$, 
we can find $z_{L,n} = Ly_m$ for some $m \in \ZZ$ 
such that $\|z_{L,n} - \tilde{y}_n\| < \eta_0$.
Therefore we obtain that
$$\|\tilde{c}(\cdot + z_{L,n}) - \tilde{c}(\cdot + \tilde{y}_n)\| 
\leq \gamma/3.$$

In conclusion, for any $L < L_0$, there exists a relatively dense set $\{z_{L,n}\}_{n \in \ZZ}$ and a constant $N > 0$, independent of $L$, such that
\begin{equation*}
\begin{aligned}
&\|A-A(\cdot+z_{L,n})\| 
+ \|B(\cdot)-B(\cdot+z_{L,n})\| \\
&+ \|c_L(\cdot)-c_L(\cdot+z_{L,n})\| 
+ \|\tilde{c}(\cdot) - \tilde{c}(\cdot + z_{L,n})\| 
\leq \gamma
\end{aligned}
\end{equation*}
and $\bigcup\limits_{n \in \ZZ}B(z_{L,n},N) = \RR^n$.

Define the translated coefficients by $a_{L,n}^{ij}(x) = a_L^{ij}(x+z_{L,n})$,
with $b_{L,n}^i, c_{L,n}$ defined similarly. 
Denote $\tilde{c}_{L,n}(x) = \tilde{c}(x+z_{L,n})$, then one has 
\begin{equation*}
\begin{aligned}
&\|\epsilon u_L 
- \sum\limits_{i,j=1}^{n}a_{L,n}^{ij}D_{ij}u_L 
- \sum\limits_{i,j=1}^{n}a_{L,n}^{ij}(D_iu_L-p_i)(D_ju_L-p_j)
- \sum\limits_{i=1}^{n}b_{L,n}^i(D_iu_L-p_i) 
- c_{L,n} 
- \tilde{c}_{L,n}\| \\
&= \|\sum\limits_{i,j=1}^{n}(a_L^{ij} - a_{L,n}^{ij})D_{ij}u_L
+ \sum\limits_{i,j=1}^{n}(a_L^{ij}-a_{L,n}^{ij})(D_iu_L-p_i)(D_ju_L-p_j) \\
&\quad + \sum\limits_{i=1}^{n}(b_L^i-b_{L,n}^i)(D_iu_L-p_i) 
+ (c_L-c_{L,n}) 
+ (\tilde{c} 
- \tilde{c}_{L,n})\| \\
&\leq C_p (\sum\limits_{i,j=1}^{n}\|a_L^{ij}-a_{L,n}^{ij}\| 
+ \sum\limits_{i=1}^{n}\|b_L^i-b_{L,n}^i\| 
+ \|c_L-c_{L,n}\| 
+ \|\tilde{c}- \tilde{c}_{L,n}\|) \\
&\leq C_p \gamma.
\end{aligned}
\end{equation*}
Here, the constant $C_p$ is independent of $L, \epsilon$ and $\gamma$. 
Comparison principle \cite{crandall1992user} yields that
$$\|\epsilon u_L(\cdot) - \epsilon u_L(\cdot+z_{L,n})\| \leq C_p\gamma.$$
For any $x,y \in \RR^n$, we can choose $n \in \ZZ$ such that $|x-z_{L,n}| \leq N$.
It follows that
\begin{equation*}
\begin{aligned}
|\epsilon u_L(x) - \epsilon u_L(y)| 
&\leq |\epsilon u_L(x) - \epsilon u_L(z_{L,n})| 
+ |\epsilon u_L(z_{L,n}) - \epsilon u_L(y)| \\
& \leq \epsilon \|Du_L\||x-z_{L,n}| 
+ C_p\gamma 
\leq 2 C_p \gamma
\end{aligned}
\end{equation*}
provided $\epsilon$ is sufficiently small.
By \thref{characterization of effective Ham}, we therefore obtain that 
\begin{equation*}
\|\epsilon u_L - \lambda(L,p)\| \leq 2C_p\gamma,
\end{equation*}
which implies that $\lambda(L,p)$ converges to $\overline\lambda(p)$, the effective Hamiltonian of \eqref{zeroeigenvalues} as $L \to 0$. 
The locally uniform convergence with respect to $p \in \RR^n$ follows from the uniform boundedness of the constants $C$ and $C_p$ in the preceding estimates for $p$ in bounded sets.
\end{proof}

Now we could prove \thref{zero limit}.

\begin{proof}
[\textbf{Proof of \thref{zero limit}}] For any $L > 0$, 
by taking $v \equiv 1$ as a test function and using \eqref{spreading occurs}, 
we can find $\alpha_i, \beta_i > 0 \ (i=1,2)$ independent of $L$ such that 
\begin{equation}\label{eigenvalueestimate}
\alpha_1 |p|^2 + \beta_1 \leq \lambda(L,p) \leq \alpha_2 |p|^2 + \beta_2.
\end{equation}
Moreover, by \thref{convexity}, 
$\lambda(L,p)$ is convex in $p$.
In conclusion, 
for any $e \in \mathbb{S}^{n-1}$ and $L > 0$, 
there exists $p^L \in \RR^n$ with $p^L \cdot e > 0$ such that
$$\omega(e;L) = \frac{\lambda(L,p^L)}{p^L \cdot e}.$$
It suffices to prove that the set $\{p^L \cdot e\}_{L > 0}$ is bounded away from infinity and zero.
The upper bound in \eqref{eigenvalueestimate} implies
$$\omega(e;L) \leq \lambda(L,e) \leq \alpha_2 + \beta_2 \leq C,$$
which shows that $\omega(e;L)$ is uniformly bounded from above in $L$.
For the lower bound in \eqref{eigenvalueestimate}, we have
$$\alpha_1(p^L \cdot e)^2 
+ \beta_1 
\leq \alpha_1|p^L|^2 
+ \beta_1 
\leq \omega(e;L)p^L \cdot e 
\leq C (p^L \cdot e).$$
The desired bounds on $p^L \cdot e$ the follow from an elementary analysis of this quadratic inequality.
\end{proof}

\begin{remark}\thlabel{boundedness of p}
We furthermore have the chain of inequalities
$$\alpha_1 |p^L|^2 + \beta_1 \leq \lambda(L,p^L) = 
\omega(e;L) (p^L \cdot e) \leq (\alpha_2 + \beta_2)|p^L|.$$
This implies that $\{p^L\}_{L > 0}$ is uniformly bounded.
\end{remark}

To study the convergence rate of the generalized principal eigenvalues, 
we continue the analysis from the beginning of this section.
Let $v_0$ be the unique bounded solution of \eqref{limitat0}, and $v_2 = v_2(x,y)$ the unique bounded solution of 
$$\sum\limits_{i,j=1}^{n}a^{ij}(y)D_{y_iy_j}v_2 + F(x,y) = L^{2}v_2,$$
where
\begin{equation*}
F(x,y) = \tilde{H}(x,y,Dv_0(x) - p, D^2v_0(x)) - \epsilon v_0(x).
\end{equation*}
Define $v_L(x) = v_0(x) + L^2v_2(x,\frac{x}{L})$ and substitute it into \eqref{approximation} to obtain
\begin{equation*}
\begin{aligned}
&\quad\epsilon v_L - \tilde{H}(x,\frac{x}{L},Dv_L - p, D^2v_L) \\
&= (\epsilon - 1)L^2 v_2
+ \sum\limits_{i,j=1}^{n}( 2La_L^{ij}D_{x_iy_j}v_2 + L^2a_L^{ij}D_{x_ix_j}v_2)
-2 \sum_{i,j=1}^{n}a_L^{ij}(D_iv_0 - p_i)(LD_{y_i}v_2 + L^2D_{x_i}v_2) \\
&\quad-\sum\limits_{i,j=1}^{n}a_L^{ij}(LD_{y_i}v_2+L^2D_{x_i}v_2)(LD_{y_j}v_2 + L^2D_{x_j}v_2) 
-\sum\limits_{i=1}^{n}b_L^i(LD_{y_i}v_2+L^2D_{x_i}v_2).
\end{aligned}
\end{equation*}
Then we have
\begin{equation*}
\|\epsilon v_L - \tilde{H}(x,\frac{x}{L},Dv_L-p,D^2v_L)\| \leq C(L^2\|v_2\| + L\|D_yv_2\| + L\|D_{xy}v_2\|+ L^2\|D_xv_2\| + L^2\|D_x^2v_2\|).
\end{equation*}

%$$\begin{array}{l}
%\|\epsilon v_L 
%- \sum\limits_{i,j=1}^{n}a_L^{ij}D_{ij}v_L 
%- \sum\limits_{i,j=1}^{n}a_L^{ij}(D_iv_L-p_i)(D_jv_L-p_j) 
%- \sum\limits_{i=1}^{n}b_L^i(D_iv_L-p_i) - c_L - \tilde{c}\| \\
%\leq C(L^2\|v_2\| + L\|D_yv_2\| + L\|D_{xy}v_2\|+ L^2\|D_xv_2\| + L^2\|D_x^2v_2\|).
%\end{array}$$
To proceed,
we require estimates for $\|v_2\|$ and its derivatives.
\begin{lemma}
For any $\sigma \in (0,1)$,
there exists $C$ such that
$$L^2\|v_2\| + L^2\|D_xv_2\| + L^2\|D_x^2v_2\| + L\|D_yv_2\| + L\|D_{xy}v_2\| \leq C\Theta_\sigma(L;(A,B,c)).$$
\end{lemma}

\begin{proof}
For any $z \in \RR$, let $\varphi(x,y) = v_2(x,y) - v_2(x,y+z)$.
Then $\varphi$ solves
$$\sum\limits_{i,j=1}^{n}a^{ij}(y)D_{y_iy_j}\varphi - L^2\varphi = F(x,y+z)-F(x,y) + \hbox{tr}((A(y+z) - A(y))D_y^2v_2(x,y+z)).$$
Note that 
$$\|D^2v_0\| + \|Dv_0\| + \|\epsilon v_0\| \leq C$$
for some $C \in \RR$ independent of $\epsilon$ and $L$ by \eqref{bound}, 
by classical elliptic regularity theory, we obtain that 
\begin{equation*}
\begin{aligned}
&L^2\|v_2(x,\cdot) - v_2(x,\cdot+z)\| 
+ L \|D_yv_2(\cdot) - D_yv_2(\cdot+z)\| \\
&\leq C(\|A(\cdot) - A(\cdot+z)\| 
+ \|B(\cdot) + B(\cdot+z)\| 
+ \|c(\cdot) + c(\cdot+z)\|).  
\end{aligned}
\end{equation*}
For any $R > 0$, there exists $\tilde{z} \in B(0,R)$ such that
$$\|A(\cdot+\tilde{z}) - A(\cdot+z)\| 
+ \|B(\cdot+\tilde{z}) - B(\cdot+z)\| 
+ \|c(\cdot+\tilde{z}) - c(\cdot+z)\| 
\leq \rho(R;(A,B,c)).$$
Therefore we have
\begin{equation*}
\begin{aligned}
L^2\|v_2(x,\cdot) - v_2(x,\cdot + z)\|  
&\leq L^2\|v_2(x,\cdot) - v_2(x,\cdot+\tilde{z})\| 
+ L^2\|v_2(x,\cdot+\tilde{z}) - v_2(x,\cdot+z)\|\\
&\leq CLR + C\rho(R;(A,B,c)). 
\end{aligned}
\end{equation*}
This yields that
$$L^2\|v_2(x,\cdot) - v_2(x,\cdot+z)\| \leq C\Theta_1(L;(A,B,c))$$
for any $z \in \RR^n$.
Thus for any $x,y,z \in \RR^n$, we conclude that
$$L^2|v_2(x,y) - v_2(x,z)| \leq C\Theta_1(L;(A,B,c)). $$
Note that by \thref{characterization of iota}, 
for any $x \in \RR$, 
there always exists $z \in \RR^n$ such that $|v_2(x,z)| < \Theta_1(L;(A,B,c))$, 
and therefore the conclusion follows for $L^2\|v_2\|$.

Similarly,
for $L\|D_y v_2\|$,
we have
\begin{equation*}
\begin{aligned}
L\|D_yv_2(x,\cdot) - D_yv_2(x,\cdot + z)\|  
&\leq L\|D_yv_2(x,\cdot) - D_yv_2(x,\cdot+\tilde{z})\|  \\
&\quad + L\|D_yv_2(x,\cdot+\tilde{z}) - D_yv_2(x,\cdot+z)\|\\
&\leq C_\sigma(LR)^\sigma + C\rho(R;(A,B,c)),
\end{aligned}
\end{equation*}
where we have used standard elliptic regularity theory.
Thus we have
$$L|D_yv_2(x,y) - D_yv_2(x,z)| \leq C_\sigma\Theta_\sigma(L;(A,B,c))$$
for any $x,y,z \in \RR^n$.
Note that $v_2$ is bounded,
which implies that for any $x \in \RR^n$, 
there exists $z \in \RR^n$ such that $|D_yv_2(x,z)| \leq \Theta_\sigma(L;(A,B,c))$,
therefore we finally obtain that
$$L\|D_yv_2\| \leq C_\sigma \Theta_\sigma(L;(A,B,c)).$$

The remaining terms can be estimated by similar arguments, which we omit here.
\end{proof}

Using the preceding estimates and the comparison principle,
we obtain
$$\epsilon\|u_L - v_L\| \leq C_\sigma\Theta_\sigma(L;(A,B,c)),$$
This lead to the main convergence rate results:

\begin{theorem}\thlabel{ratezero}
For any $p \in \RR^n$, we have
$$|\lambda(L,p) - \overline\lambda(p)| \leq C_\sigma\Theta_\sigma(L;(A,B,c)),$$
where $C_\sigma$ depends only on $\sigma, p$ and the coefficients of \eqref{equation}.
\end{theorem}

\begin{proof}
We have
\begin{equation*}
\begin{aligned}
|\lambda(L,p) -\overline \lambda(p)|
&\leq \|\lambda(L,p) - \epsilon u_L\|
+ \epsilon\|u_L - v_L\|
+ \epsilon \|v_L - v\|
+ \|\epsilon v - \overline\lambda(p)\| \\
&\leq C_\sigma \Theta_\sigma(L;(A,B,c)) 
+ \|\lambda(L,p) - \epsilon u_L\|
+ \|\epsilon v -\overline \lambda(p)\|.
\end{aligned}
\end{equation*}
The result follows by letting $\epsilon \to 0$.
\end{proof}

\begin{proof}
[\textbf{Proof of \thref{zerorate}}]
We have shown that for any $L > 0$ and $e \in \mathbb{S}^{n-1}$ there exists $p^L$ satisfying
$$\omega(e;L) = \frac{\lambda(L,p^L)}{p^L \cdot e}$$
with uniform bounds
$$0 < m < p^L \cdot e < M < \infty, \quad |p^L| \leq M$$
for all $L > 0$ and some $m, M \in \RR$.
For any $p \in \RR^n$ with $m < p \cdot e < M$, 
the preceding results imply
$$\omega(e;0) (p \cdot e) \leq\overline \lambda(p) \leq \lambda(L,p) + C_\sigma\Theta_\sigma(L;(A,B,c))$$
for some $s \in (0,1]$.
Since $m < p \cdot e < M$, 
it follows that
$$\omega(e;0) \leq \frac{\lambda(L,p)}{p \cdot e} + C_\sigma\Theta_\sigma(L;(A,B,c)).$$
Taking $p = p^L$ gives
$$\omega(e;0) \leq \omega(e;L) + C_\sigma\Theta_\sigma(L;(A,B,c)).$$
Conversely, from the inequality
$$\omega(e;L) (p \cdot e) \leq \lambda(L,p) \leq \overline\lambda(p) + C_\sigma\Theta_\sigma(L;(A,B,c)),$$
we obtain the reverse inequality.

In the periodic case,
one readily shows that $\Theta_\sigma(L;(A,B,c)) \leq CL^\sigma$.
\end{proof}

\section{Limit behavior when $L \to \infty$: Proof of \thref{maintheoreminfinity,infity limit,infinityrate}}

In this section, we prove the limit behavior of the generalized principal eigenvalue $\lambda(L,p)$ as $L \to \infty$.

We begin by approximating the generalized principal eigenvalues of $\Lcal(L,p)$ via a scaled Hamiltonian.

\begin{proposition}
The generalized principal eigenvalue $\lambda(L,p)$ of $\Lcal(L,p)$ coincides with the effective Hamiltonian of 
\begin{equation}\label{Hamiltoninfinity}
\tilde{H}(Lx,x,Dv-p,\frac{1}{L}D^2v) = \lambda.
\end{equation}
\end{proposition}

\begin{proof}
For any $\epsilon > 0$,
let $u_\epsilon$ ba a bounded and Lipschitz continuous viscosity solution of
$$\tilde{H}(x,\frac{x}{L},Du_\epsilon-p,D^2u_\epsilon) = \epsilon u_\epsilon.$$
Define the scaled function $v_\epsilon(x) = u_\epsilon(Lx)/L$, then $v_\epsilon$ satisfies
$$\tilde{H}(Lx,x,Dv_\epsilon-p,\frac{1}{L}D^2v_\epsilon) = \epsilon L v_\epsilon.$$
Passing to the limit as $\epsilon \to 0$,
we conclude that $\lambda(L,p)$ is the effective Hamiltonian of \eqref{Hamiltoninfinity}.
\end{proof}

This provides a characterization of $\lambda(L,p)$ via \eqref{Hamiltoninfinity}.

\begin{proof}
[\textbf{Proof of \thref{maintheoreminfinity}}]
Fixed $\epsilon > 0$, 
for any $p \in \RR^n$ and $L > 0$, 
let $v_p^L$ be the unique solution of
\begin{equation}\label{corrector}
\tilde{H}(Lx,x,Du-p,\frac{1}{L}D^2u) = \epsilon u.
\end{equation} 
By \eqref{bound}, 
there exists $M$ depending only on the coefficients of the equation and $p$ such that
\begin{equation}\label{bounds}
\|\epsilon v_p^L\| + \|Dv_p^L\| + \|\frac{1}{L}D^2v_p^L\| \leq M.
\end{equation} 
Therefore, 
$\{v_p^L\}_{L > 1}$ is uniformly bounded and equicontinuous. 
By the Arzela-Ascoli theorem, 
there exists a subsequence of
$\{v_p^{L}\}_{L > 0}$ converges to some $v_p \in C(\RR^n)$ locally uniformly as $L \to \infty$.

We will show that $v_p$ is indeed an approximate corrector that characterizes the limit of the generalized principal eigenvalues.
Specifically, 
for any $x_0 \in \RR^n$, 
assume $\phi \in C^2(\RR^n)$ touches $v_p$ from below at $x = x_0$. 
To identify the equation that $v_p$ solves, 
consider the solution of the following problem
$$\tilde{H}(y,x_0,D_y\psi+q,D_y^2\psi) = \overline{H}(x_0,q),$$
where $\overline{H}(x_0,q)$ is the effective Hamiltonian and $\psi = \psi(y)$ is sublinear at infinity. 
However, as we have described before, 
there is no such solution in general and thus we let $\psi = \psi(y)$ be a $\delta$-approximate corrector of the equation, 
that is, 
\begin{equation}\label{trans}
\|\tilde{H}(\cdot,x_0,D_y\psi+q,D_y^2\psi) - \overline{H}(x_0,q)\| \leq \delta.
\end{equation}
Moreover, $\|\psi\|_{W^{2,\infty}(\RR^n)} \leq C$ for some $C > 0$.
Now let $q = D\phi(x_0) - p$ and 
$$\phi^L(x) = \phi(x) + \frac{1}{L}\psi(Lx).$$
Since $v_p^L, \phi^L$ converges to $v,\phi$ locally uniformly respectively, 
there exists $\{x_L\}_{L > 0}$ such that $v_p^L - \phi^L$ attains its minimum at $x_L$ and $x_L \to x_0$ as $L \to \infty$ \cite{tran2021hamilton}.
Thus
$$Dv_p^L(x_L) = D\phi(x_L) + D\psi(Lx_L) \quad \hbox{ and } \quad 
D^2v_p^L(x_L) \geq D^2\phi(x_L) + LD^2\psi(Lx_L).$$
Substituting these into \eqref{corrector} at $x = x_L$ gives
\begin{equation*}
\begin{aligned}
\epsilon v_p^L(x_L) 
&= \tilde{H}(Lx_L, x_L, Dv_p^L-p,\frac{1}{L}D^2v_p^L) \\
& \geq  \tilde{H}(Lx_L, x_L, D\phi(x_L) + D\psi(Lx_L) - p, \frac{1}{L}D^2\phi(x_L) + D^2\psi(Lx_L)).
\end{aligned}
\end{equation*}
Using \eqref{trans} we estimate
\begin{equation*}
\begin{aligned}
&|\tilde{H}(Lx_L,x_L,D\phi(x_L)+D\psi(Lx_L)-p, \frac{1}{L}D^2\phi(x_L) 
+ D^2\psi(Lx_L)) 
- \overline{H}(x_0,q)| \\
&\leq \frac{1}{L}|\hbox{tr}(A(x_L)D^2\phi(x_L))| 
+ C|A(x_L) - A(x_0)| 
+ |D\phi(x_L) - D\phi(x_0)| \\
&\quad + |D\phi(x_L) - D\phi(x_0)|^2 
+ |B(x_L)- B(x_0)| 
+ |c(x_L) - c(x_0)| 
+ \delta\\
&\leq \rm o(1) + \delta,
\end{aligned}
\end{equation*}
where $o(1) \to 0$ as $L \to +\infty$.
Therefore we obtain 
$$\epsilon v_p^{L}(x_L) \geq \overline{H}(x_0,D\phi-p) - o(1) - \delta.$$
Taking the limit as $L \to \infty$, followed by $\delta \to 0$, 
shows that $v_p$ is a viscosity supersolution of
\begin{equation}\label{HJ}
\epsilon v - \overline{H}(x,Dv-p) = 0.
\end{equation}
A parallel argument shows that $v_p$ is also a viscosity subsolution,
and hence a viscosity solution of \eqref{HJ}.

The remainder of the proof follows an argument similar to that of \thref{maintheoremzero},
provided we establish that
\begin{equation*}
\lim\limits_{\epsilon \to 0} \mathop{\mathrm{osc}}\limits_{\RR^n} \epsilon v_p = 0,
\end{equation*}
To prove this, 
one may adapt the method of Isshi\cite{ishii2000almost},
under the assumption that $\overline{H}$ satisfies the hypothesis of \thref{Ishii}.
We verify these conditions later in \thref{assupmtion check}.
\end{proof}

Let us gather some propositions $\overline{H}$ satisfies.

\begin{proposition}\thlabel{smoothness of the effective Hamiltonian}
For any $p \in \RR^n$,
$\overline{H}(\cdot,p) \in C^{2+\gamma}(\RR^n)$.

For any $x \in \RR^n$,
$\overline{H}(x,\cdot)$ is smooth.
\end{proposition}

\begin{proof}
For any $x,p \in \RR^n$ and $\epsilon > 0$,
let $\phi_\epsilon(y) = \phi_\epsilon(y;x,p)$ be the unique bounded and uniformly continuous solution of
\begin{equation}\label{regularity in x}
\tilde{H}(y,x,D\phi_\epsilon + p, D^2\phi_\epsilon) = \epsilon \phi_\epsilon.
\end{equation}
For a fixed $z \in \RR^n$,
define $\tilde{\phi}_\epsilon(y) = \phi_\epsilon(y;x+z,p)$,
and set $\varphi = \phi_\epsilon - \tilde{\phi}_\epsilon$.
Then $\varphi$ satisfies
\begin{equation*}
\begin{aligned}
&\sum\limits_{i,j=1}^{n}a^{ij}(x)D_{ij}\varphi 
+ \sum\limits_{i,j=1}^{n}a^{ij}(x)D_i\varphi D_j\varphi  \\ 
&+ 2a^{ij}(x)(D_i\tilde{\phi}_\epsilon + p_i)D_j\varphi 
+ \sum\limits_{i,j=1}^{n}b^i(x)D_i\varphi - \epsilon \varphi = F(x,y),
\end{aligned}
\end{equation*}
where 
\begin{equation*}
\begin{aligned}
F(x,y) = &\sum\limits_{i,j=1}^{n}(a^{ij}(x+z) - a^{ij}(x))(D_{ij}\tilde{\phi}_\epsilon + (D_i\tilde{\phi}_\epsilon + p_i)(D_j\tilde{\phi}_\epsilon + p_j)) \\
&+ \sum\limits_{i=1}^{n}(b^i(x+z) - b^i(x))(D_i\tilde{\phi}_\epsilon + p_i) + c(x+z) - c(x).
\end{aligned}
\end{equation*}
From the regularity estimate of $\tilde{\phi}_\epsilon$,
it follows that 
\begin{equation*}
\|F\| \leq C(\|A(x+z) - A(x)\| + \|B(x+z) - B(x)\| + \|c(x+z) - c(x)\|).
\end{equation*}
By regularity theory and uniformly continuity of $A,B,c$,
we have
\begin{equation}\label{Lipschitz in x}
\|D^2(\phi_\epsilon - \tilde{\phi}_\epsilon)\| + \|D(\phi_\epsilon - \tilde{\phi}_\epsilon)\| + \epsilon \|(\phi_\epsilon - \tilde{\phi}_\epsilon)\| \leq \|F\| \leq C|z|.
\end{equation}
In particular,
\begin{equation*}
\|\epsilon \phi_\epsilon - \epsilon \tilde{\phi}_\epsilon\| \leq C|z|.
\end{equation*}
Passing to the limit as $\epsilon \to 0$ yields
\begin{equation*}
|\overline{H}(x,p) - \overline{H}(x+z,p)| \leq C|z|.
\end{equation*}
Since the constant $C$ is bounded locally uniformly in $p$,
we conclude that $\overline{H}$ is uniformly Lipschitz continuous in $x$, locally uniformly in $p$.

Differentiating \eqref{regularity in x} with respect to $x_k$,
we find that $\varphi(y) = D_{x_k}\phi_\epsilon(y;x,p)$ satisfies
\begin{equation*}
\sum\limits_{i,j=1}^{n}a^{ij}(x)D_{ij}\varphi 
+ \sum\limits_{i,j=1}^{n}2a^{ij}(x)(D_i\phi_\epsilon + p_i)D_j\varphi 
+ \sum\limits_{i=1}^{n}b^i(x)D_i\varphi - \epsilon \varphi = K(x,y),
\end{equation*}
where 
\begin{equation*}
\begin{aligned}
K(x,y) 
&= -\sum\limits_{i,j=1}^{n}D_ka^{ij}(x)D_{ij}\phi_\epsilon 
- \sum\limits_{i,j=1}^{n}D_ka^{ij}(x)(D_i\phi_\epsilon + p_i)(D_j\phi_\epsilon + p_j) \\
&\quad - \sum\limits_{i=1}^{n}D_kb^i(x)D_i\phi_\epsilon - D_kc(x).
\end{aligned}
\end{equation*}
The source term $K$ is bounded.
Applying \eqref{Lipschitz in x},
we obtain
\begin{equation*}
\|D^2\varphi\| + \|D\varphi\| + \|\epsilon \varphi\| \leq C.
\end{equation*}
By the uniform almost periodicity of $K(x,y)$ in $y$ (established in\eqref{uniform a.p. of corrector} below) 
and the arguments in \cite{lions2005homogenization},
it follows that
\begin{equation*}
\lim\limits_{\epsilon \to 0}\mathop{\mathrm{osc}}\limits_{\RR^n} \ \epsilon\varphi = 0.
\end{equation*}
Consequently, 
the limit $\lim_{\epsilon \to 0}\epsilon D_{x_k}\phi_\epsilon(y;x,p)$ exists and equals $D_{x_k}\overline{H}(x,p)$,
establishing the differentiability of $\overline{H}$ in $x$.

Higher regularity in $x$ and $p$ can be established by similar arguments involving
higher order difference of $\phi_\epsilon$.
\end{proof}

\begin{remark}
By fixing $L$, we also proved \thref{convexity} here. From the proof of this proposition, we can also obtained that the derivatives of $\overline{H}$ are bounded uniformly in $x$, locally in $p$,
provided the coefficients are smooth with bounded derivatives.
\end{remark}

\begin{proposition}\thlabel{assupmtion check}
$\overline{H}$ satisfies the conditions stated in \thref{Ishii}.
\end{proposition}

\begin{proof}
By \thref{characterization of effective Ham},
if we take $\psi \equiv 1$ as a test function,
then we see that there exists $0< \alpha_1 < \alpha_2 < +\infty$ and $C > 0$ such that
$$\alpha_1 |p|^2 - C \leq \overline{H}(x,p) \leq \alpha_2 |p|^2 + C,$$
therefore $\overline{H}$ satisfies (1) and (3).
(2) is implied by \thref{smoothness of the effective Hamiltonian},
(4) follows from the convexity of $\tilde{H}$ in $p$,
and (5) is a consequence of \eqref{Lipschitz in x}.
\end{proof}

We now turn to the convergence rate of the generalized principal eigenvalues. 
For this purpose,
we introduce the family of functions $\psi_\delta = \psi_\delta(y;x,p)$,
defined as solutions to
\begin{equation}\label{approxcorrector}
\tilde{H}(y,x,D_y\psi_\delta+p,D_y^2\psi_\delta) = \delta \psi_\delta + \overline{H}(x,p).
\end{equation}

First, we have

\begin{lemma}\thlabel{oscillationestimate}
There exists $C(p) > 0$ such that
$$\mathop{\mathrm{osc}}\limits_{\RR^n}\ \delta\psi_\delta(\cdot;x,p) \leq C(p)\Theta_1(\delta;\tilde{c})$$
for all $\delta \in (0,1)$ and $x \in \RR^n$.
\end{lemma}

\begin{proof}
Without loss of generality, it suffices to prove that 
\begin{equation*}
|\delta\psi_\delta(y;x,p) - \delta\psi_\delta(0;x,p)| \leq C(p)\Theta_1(\delta;\tilde{c})
\end{equation*}
for all $y \in \RR^n$.
From Bernstein's method \cite{barles2021local}, there exists $C(p) > 0$ such that for all $\delta \in (0,1)$ such that
\begin{equation}\label{bound3}
\|D_y^2\psi_\delta(\cdot;x,p)\| + \|D_y\psi_\delta(\cdot;x,p)\| + \|\delta \psi_\delta(\cdot;x,p)\| \leq C(p).
\end{equation}
Denote $\psi_\delta(y) = \psi_\delta(y;x,p)$ for simplicity, and let $\varphi(y) = \psi_\delta(y+z_1) - \psi_\delta(y+z_2)$ for any $z_1, z_2 \in \RR^n$.
Then a direct calculus computes that $\varphi$ satisfies the following equation:
\begin{equation*}
\begin{aligned}
&\sum\limits_{i,j=1}^{n}a^{ij}(x)D_{ij}\varphi(y) 
+ a^{ij}(x)D_i\varphi(y)D_j\varphi(y) 
+ 2a^{ij}(x)(D_i\psi_\delta(y+z_2) + p_i)D_j\varphi(y)\\
&+ \sum\limits_{i=1}^{n}b^i(x)D_i\varphi(y) - \delta\varphi(y) 
= \tilde{c}(y+z_2) - \tilde{c}(y+z_1).
\end{aligned}
\end{equation*} 
By comparison principle, one has
$$\|\delta\psi_\delta(\cdot+z_1) - \delta\psi_\delta(\cdot + z_2)\| \leq \|\tilde{c}(\cdot + z_1) - \tilde{c}(\cdot + z_2)\|$$
for all $z_1, z_2 \in \RR^n$.
Now for any $y \in \RR^n$, there exists $|z| < R$ such that
$$|\delta\psi_\delta(y) - \delta\psi_\delta(z)| \leq \|\tilde{c}(\cdot + y) - \tilde{c}(\cdot + z)\| \leq \rho(R;\tilde{c}).$$
This implies that
$$|\delta\psi_\delta(y) - \delta\psi_\delta(0)| \leq |\delta\psi_\delta(y) - \delta\psi_\delta(z)| + |\delta\psi_\delta(z) - \delta\psi_\delta(0)| \leq \rho(R;\tilde{c}) + \delta C(p)R.$$
Thus the conclusion follows from the definition of $\Theta_1(\delta;\tilde{c})$.
\end{proof}

Combine with \thref{characterization of effective Ham}, we have
\begin{equation}\label{correctorconverrate}
\|\delta\psi_\delta(\cdot;x,p)\| \leq C(p)\Theta_1(\delta;\tilde{c}).
\end{equation}

\begin{remark}
The preceding arguments imply that $\{\delta\psi_\delta\}_{\delta > 0}$ is uniformly almost periodic in $y$ uniformly in $x$ and locally uniformly in $p$. 
In fact, Bernstein's method \cite{barles2021local} yields the estimates
\begin{equation}\label{uniform a.p. of corrector}
\begin{aligned}
&\|D_y^2\psi_\delta(\cdot + z_1;x,p) - D_y^2\psi_\delta(\cdot + z_2;x,p)\| 
+ \|D_y\psi_\delta(\cdot + z_1;x,p) - D_y\psi_\delta(\cdot + z_2;x,p)\| \\
&+ \|\delta\psi_\delta(\cdot + z_1;x,p) - \delta\psi_\delta(\cdot + z_2;x,p)\| 
\leq C(p)\|\tilde{c}(\cdot + z_1) - \tilde{c}(\cdot + z_2)\|.
\end{aligned}
\end{equation} 
\end{remark}

\begin{lemma}\thlabel{correctorestimate}
$\psi_\delta(y;x,p)$ is twice differentiable in $y$ and $x$ with estimates
$$\|D_y\psi_\delta\| + \|D_y^2\psi_\delta\| + \|D_{xy}\psi_\delta\| \leq C(p)$$
and
$$\|\delta D_x\psi_\delta\| + \|\delta D_x^2\psi_\delta\| \leq C(p)$$
holding for some $C(p) > 0$ and all $\delta \in (0,1)$.
Furthermore, we have
\begin{equation}\label{Lipschitzlikeestimate}
\|\delta\psi_\delta(\cdot;x,p) - \delta\psi_\delta(\cdot;x,q)\| \leq C(p,q)\Theta_1(\delta;\tilde{c})|p-q|
\end{equation}
for all $x \in \RR$.
\end{lemma} 

\begin{proof}
The regularity and derivatives estimates of $\psi_\delta$ follow from estimates \eqref{bound3} and the proof of \thref{smoothness of the effective Hamiltonian}.

Now we prove \eqref{Lipschitzlikeestimate}.
For any $p,q \in \RR^n$ and $\delta \in (0,1)$, 
the comparison principle implies there exists $C = C(p,q) > 0$ such that
$$\|\delta \psi_\delta(\cdot;x,p) - \delta \psi_\delta(\cdot;x,q)\| \leq C(p,q)|p-q|$$
for all $x \in \RR^n$. 
Let $\varphi(y) = \psi_\delta(y;x,p) - \psi_\delta(y;x,q) + (\overline{H}(x,p)- \overline{H}(x,q))/\delta$ and consider the following equation
$$\delta\varphi = \tilde{H}(y,x,D\psi_\delta(\cdot;\cdot,p)+p,D^2\psi_\delta(\cdot;\cdot,p)) - \tilde{H}(y,x,D\psi_\delta(\cdot;\cdot,q)+q,D^2\psi_\delta(\cdot;\cdot,q)).$$
Writing this equation explicitly yields
\begin{equation}\label{ODEestimate}
\begin{aligned}
&\quad -\sum\limits_{i,j=1}^{n}a^{ij}(x)D_{ij}\varphi 
- \sum\limits_{i,j=1}^{n}a^{ij}(x)D_i\varphi D_j\varphi 
- \sum\limits_{i=1}^{n}N^i(x,y)D_i\varphi + \delta \varphi \\
&= \sum\limits_{i,j=1}^{n}a^{ij}(x)(p_i-q_i)(p_j-q_j) 
+ \sum\limits_{i=1}^{n}b^i(x)(p_i-q_i),
\end{aligned}
\end{equation}
where 
$$N^i(x,y) 
= 2\sum\limits_{i,j=1}^{n}a^{ij}(x)(D_j\psi_\delta(y;x,q)+p_j) 
+ \sum\limits_{i=1}^{n}b^i(x).$$
By comparison principle, we have
$$\|\delta \varphi\| \leq C(p,q)|p-q|.$$
Letting $\delta \to 0$, we obtain
$$|\overline{H}(x,p) - \overline{H}(x,q)| \leq C(p,q)|p-q|.$$
Then by Bernstein's method \cite{barles2021local} and \eqref{uniform a.p. of corrector} we have
$$\|D\varphi\| + \|\delta\varphi\| \leq C(p,q)|p-q|,$$
and
$$\|\delta\varphi(\cdot + z_1) - \delta\varphi(\cdot + z_2)\| \leq C(p,q)\|\tilde{c}(\cdot + z_1) - \tilde{c}(\cdot + z_2)\||p-q|.$$
Define $\phi(y) = \psi_\delta(y;x,p) - \psi_\delta(y;x,q)$, then for any $y,z \in \RR^n$, one can choose $w \in B(z,R)$ such that
\begin{equation*}
\begin{aligned}
|\delta \phi(y) - \delta \phi(z)| 
& \leq &&|\delta\phi(y) - \delta\phi(w)| + |\delta\phi(w) - \delta\phi(z)| \\
& = &&|\delta \varphi(y) - \delta \varphi(w)| + |\delta\varphi(w) - \delta\varphi(z)| \\
& \leq && C(p,q)|p-q|(\rho(R;\tilde{c}) + \delta R).
\end{aligned}
\end{equation*}
From \thref{characterization of effective Ham},
we may choose $z$ such that $|\delta \phi(z)|$ is arbitrarily small.
Then the estimate \eqref{Lipschitzlikeestimate} follows.
\end{proof}

If $\tilde{c}$ satisfies \eqref{a.p.char},
then choosing $R = (\frac{\delta}{C})^{-\frac{1}{\tau+1}}$ yields
$$\Theta_1(\delta;\tilde{c}) \leq \rho(R;\tilde{c}) + \delta R \leq C\delta^{\frac{\tau}{\tau+1}}.$$
Consequently,
estimates \eqref{correctorconverrate} and \eqref{Lipschitzlikeestimate} simplify to
\begin{equation}
\|\delta\psi_\delta(\cdot;x,p)\| \leq C(p)\delta^{\frac{\tau}{\tau+1}}
\end{equation}
and
\begin{equation}\label{Lipschitzestimate}
\delta^{\frac{1}{\tau+1}}\|\psi_\delta(\cdot;x,p) - \psi_\delta(\cdot;x,q)\| \leq C(p,q)|p-q|
\end{equation}
respectively. 
Following the approach of \cite{qian2024optimal} we have:

\begin{theorem}\thlabel{convergence rate corrector}
Assume \eqref{a.p.char} holds for some $C > 0$ and $\tau > 0$, then 
$$\|\delta v_p^L - \delta v_p\| \leq C(p)L^{\frac{-\tau}{2\tau+1}}$$
for some $C(p) \in \RR$.
\end{theorem}

\begin{proof}
The idea is adapted from \cite{qian2024optimal},
so we only sketch the proof.
Note that by \eqref{bounds}, we may restrict attention to bounded values of $p$. 
Thus, the results are unaffected if we redefine $\tilde{H}(x,y,p,X) = \sum_{i,j=1}^{n}a^{ij}(y)X_{ij} + |p|^2$ for sufficiently large $|p|$. 
Under this modification, we have $\psi_\delta(\cdot;x,p) \equiv 0$ for all $x \in \RR^n$ and $\delta > 0$ whenever $|p|$ is large enough. 
This allows us to remove the dependence of the constant in (large) $p$ and $q$ appeared in the preceding estimates.
Therefore, when fixing $p$, we could assume without loss of generality that
$$\begin{array}{l}
\mathop{\mathrm{osc}}\limits_{\RR^n}\ \delta\psi_\delta(\cdot;x,q) \leq \tilde{\Delta}_{C(p)}(\delta), \\
\|D_y\psi_\delta(\cdot;x,q)\| + \|D_y^2\psi_\delta(\cdot;x,q)\| + \|D_{xy}\psi_\delta(\cdot;x,q)\| \leq C(p), \\
\|\delta D_x\psi_\delta(\cdot;x,q)\| + \|\delta D_x^2\psi_\delta(\cdot;x,q)\| \leq C(p), \\
\|\delta\psi_\delta(\cdot;x,q) - \delta\psi_\delta(\cdot;x,\tilde{q})\| \leq C(p)\tilde{\Delta}_1(\delta)|q-\tilde{q}|
\end{array}$$
for all $\delta > 0$ and $x,q,\tilde{q} \in \RR^n$.

We will adopt some notations from the proof of \thref{maintheoreminfinity}.
As the case $p=0$ involves no essential difference, 
we present the proof only for this case and henceforth omit the subscript from $v_0^L$ and $v_0$ in what follows.

Define the auxiliary function
\begin{equation*}
\begin{aligned}
\Phi (x,y,z) & := \epsilon v^L(x) - \epsilon v(y) - \psi_\delta(Lx;x,2L^{\beta}(z-y))/L \\
& - L^{\beta}|x-y|^2 - L^{\beta}|x-z|^2 - \eta \sqrt{1+|x|^2},
\end{aligned}
\end{equation*}
where $\eta \in (0,1)$,$\beta = \tau/(2\tau+1)$ and $\delta = L^{-\theta}$ with $\theta = (\tau+1)/(2\tau+1)$.
From the boundedness of $v^L, v$ and $\psi_\delta$, we could assume that $\Phi$ attains its maximum in some point $(\hat{x},\hat{y},\hat{z})$. 
Note that we could assume that this maximum is strict by adding to $\Phi$ a smooth function vanishing together with its first and second derivatives at $(\hat{x},\hat{y},\hat{z})$.

From $\Phi(\hat{x},\hat{y},\hat{z}) \geq \Phi(0,0,0)$, we have 
\begin{equation*}
\begin{aligned}
\eta\sqrt{1+|\hat{x}|^2} 
& \leq \epsilon(v^L(\hat{x})-v^L(0)) 
- \epsilon(v(\hat{y}) - v(0)) 
- L^{\beta}|\hat{x}-\hat{y}|^2 
- L^{\beta}|\hat{x}-\hat{z}|^2 
+ \eta \\
&\quad - (\psi_\delta(L\hat{x};\hat{x},2L^{\beta}(\hat{z} - \hat{y}))
- \psi_\delta(L\hat{x};\hat{x},0))/L 
- (\psi_\delta(L\hat{x};\hat{x},0) 
- \psi_\delta(0;0,0))/L \\
& \leq C 
- L^{\beta}(|\hat{x}-\hat{y}|^2 
+ |\hat{x} - \hat{z}|^2) 
+ CL^{\beta+\frac{\theta}{\tau+1}-1}(|\hat{z}-\hat{x}| 
+ |\hat{x}-\hat{y}|) 
+ CL^{\theta-1} \\
& \leq C 
+ CL^{\beta + \frac{2\theta}{\tau+1} - 2} 
+ C L^{\theta-1} 
\leq C
\end{aligned}
\end{equation*}
for $L \geq 1$. Here we have used Young's inequality in the last second inequality.

Inequality $\Phi(\hat{x},\hat{y},\hat{z}) \geq \Phi(\hat{x},\hat{y},\hat{x})$ gives 
\begin{equation*}
\begin{aligned}
L^{\beta}|\hat{x}-\hat{z}|^2 
&\leq (\psi_\delta(L\hat{x};\hat{x},2L^{\beta}(\hat{z}-\hat{y}))-\psi_\delta(L\hat{x};\hat{x},2L^{\beta}(\hat{x}-\hat{y})))/L \\
&\leq CL^{\beta}|\hat{x}-\hat{z}|/L^{1-\frac{\theta}{\tau+1}}.
\end{aligned}
\end{equation*}
Therefore we obtain that
$$|\hat{x}-\hat{z}| \leq CL^{\frac{\theta}{\tau+1}-1} = CL^{-2\beta}.$$

Moreover, from inequality $\Phi(\hat{x},\hat{y},\hat{z}) \geq \Phi(\hat{x},\hat{x},\hat{z})$ we obtain that
$$|\hat{x}-\hat{y}| \leq CL^{-\beta}$$
for $L \geq 1$ by similar arguments.
The above inequalities tell that
$$|\hat{y}-\hat{z}| \leq CL^{-\beta}.$$
which implies that $L^{\beta}|\hat{y}-\hat{z}|$ is uniformly bounded in $L$.

We claim that
\begin{equation}\label{upperbound}
\epsilon v^L(\hat{x}) \leq \overline{H}(\hat{x},2L^{\beta}(\hat{z}-\hat{y})) + CL^{-\beta} + C\eta,
\end{equation}
where $C$ only depends on $\tilde{H}$.
Indeed, observe that by the definition of $\Phi$, the function
$$x \mapsto \Phi(x,\hat{y},\hat{z})$$
has a maximum at $\hat{x}$. 
Using $v^L(x) - \epsilon^{-1}\Phi(x,\hat{y},\hat{z})$ as a subsolution at $\hat{x}$ to 
$$\epsilon v^L - \tilde{H}(Lx,x,Dv^L,D^2v^L/L) = 0,$$
then we have
\begin{equation*}
\begin{aligned}
\epsilon v^L(\hat{x}) 
&\leq \tilde{H}(L\hat{x},\hat{x},D_y\psi_\delta 
+ D_x\psi_\delta/L 
+ 2L^{\beta}(\hat{x}-\hat{y}) 
+ 2L^{\beta}(\hat{x}-\hat{z}) 
+ \eta \hat{x}/\sqrt{1+|\hat{x}|^2}, \\
&\quad D_y^2\psi_\delta + 2D_{xy}\psi_\delta/L 
+ D_x^2\psi_\delta/L^2 
+ 4L^{\beta-1} 
+ \frac{\eta}{(1+|x|^2)^{\frac{3}{2}}L}) \\
&\leq \delta \psi_\delta(L\hat{x};\hat{x},2L^{\beta}(\hat{z}-\hat{y})) 
+ \overline{H}(\hat{x},2L^{\beta}(\hat{z}-\hat{y})) 
+ 4CL^{\beta}(|\hat{x}-\hat{z}|) 
+ C\eta \\
&\quad + C|D_x\psi_\delta|/L 
+ C|D_x^2\psi_\delta|/L^2 
+ C|D_{xy}\psi_\delta|/L 
+ CL^{\beta-1} 
+ C\eta/L\\
&\leq \overline{H}(\hat{x},2L^{\beta}(\hat{z}-\hat{y})) 
+ \delta^{\frac{\tau}{\tau+1}} 
+ CL^{-\beta} 
+ C\eta 
+ CL^{\theta-1} 
+ CL^{\beta-2} 
+ CL^{\beta-1} \\
&\leq \overline{H}(\hat{x},2L^{\beta}(\hat{z}-\hat{y})) 
+ CL^{-\beta} 
+ C\eta.
\end{aligned}
\end{equation*}
Next we claim that
\begin{equation}\label{lowerbound}
\epsilon v(\hat{y}) \geq \overline{H}(2L^{\beta}(\hat{x}-\hat{y}),\hat{y}) - CL^{-\beta}.
\end{equation}  
Introduce
$$\Psi(y,\xi) := \epsilon v(y) + \psi_\delta(L\hat{x};\hat{x},2L^{\beta}(\hat{z}-\xi))/L + L^{\beta}|\hat{x}-y|^2 + \alpha|y-\xi|^2 + \eta|y-\hat{y}|^2.$$
Then there exists $y_\alpha,\xi_{\alpha} \in \RR^n$ such that the function $\Psi$ has a global minimuum at the point $(y_\alpha,\xi_{\alpha})$.

Using inequalities $\Psi(y_\alpha,\xi_{\alpha}) \leq \Psi(y_\alpha,y_\alpha), \Psi(y_\alpha,\xi_{\alpha}) \leq \Psi(\hat{x},\hat{x})$ we obtain that
$$\alpha|y_\alpha-\xi_{\alpha}| \leq CL^{-\beta} \quad \hbox{ and } \quad |\hat{x}-y_\alpha| \leq CL^{-\beta}/\sqrt{1+1/\alpha}$$
respectively. 
Moreover, 
from the above estimates one could find that as $\alpha \to +\infty$, 
$(y_\alpha,\xi_{\alpha})$ converges, 
up to subsequence, 
to $(\beta,\beta)$ for some $\beta \in \RR^n$ as $\alpha \to +\infty$. 
Note that $\Psi(\beta,\beta) = \lim\limits_{\alpha \to +\infty}\Psi(y_\alpha,\xi_{\alpha}) \leq \Psi(\hat{y},\hat{y})$, 
and $(\hat{y},\hat{y})$ is a strict global minimum of the map $y \mapsto \Psi(y,y)$, 
we conclude that $(y_\alpha,\xi_{\alpha}) \to (\hat{y},\hat{y})$ as $\alpha \to +\infty$. 
Define $\phi: \RR^n \to \RR$ by $\phi(y) = \epsilon v(y) - \Psi(y,\xi_{\alpha})$, 
we see that $\epsilon v-\phi$ has a global minimum at the point $y = y_\alpha$ and therefore we compute that
$$D\phi(y_{\alpha}) = 2L^{\beta}(\hat{x}-y_\alpha) + 2\alpha(\xi_{\alpha}-y_\alpha) + 2\eta(\hat{y}-y_\alpha).$$
By the viscosity supersolution test, we have
\begin{equation*}
\begin{aligned}
\epsilon v(y_\alpha) 
&\geq \overline{H}(y_\alpha,2L^{\beta}(\hat{x}-y_\alpha) 
+ 2\alpha(\xi_{\alpha}-y_\alpha)+2\eta(\hat{y}-y_\alpha)) \\
&\geq \overline{H}(\hat{y},2L^{\beta}(\hat{x}-\hat{y})) 
- C|\hat{y} - y_\alpha| 
- CL^{\beta}|y_\alpha-\hat{y}| 
- C\alpha|y_\alpha-\xi_{\alpha}| \\
&\geq \overline{H}(\hat{y},2L^{\beta}(\hat{x}-\hat{y})) 
-C L^{-\beta} 
- C(1+L^\beta)|y_\alpha - \hat{y}|.
\end{aligned}
\end{equation*}
Then \eqref{lowerbound} follows from letting $\alpha \to +\infty$.

Now combine \eqref{upperbound} and \eqref{lowerbound} we obtain that
\begin{equation*}
\begin{aligned}
\epsilon v^L(\hat{x}) - \epsilon v(\hat{y}) 
&\leq \overline{H}(\hat{x},2L^{\beta}(\hat{z}-\hat{y})) 
- \overline{H}(\hat{y},2L^{\beta}(\hat{x}-\hat{y})) 
+ CL^{-\beta} 
+ C\eta \\
&\leq CL^{\beta}|\hat{z}-\hat{x}| 
+ C|\hat{x}-\hat{y}| 
+ CL^{-\beta} 
+ C\eta \\
&\leq CL^{-\beta} 
+ C\eta.
\end{aligned}
\end{equation*}
Using the fact that $\Phi(x,x,x) \leq \Phi(\hat{x},\hat{y},\hat{z})$ we obtain
\begin{equation*}
\begin{aligned}
\epsilon v^L(x) - \epsilon v(x) 
&\leq CL^{-\beta} 
- \psi_\delta(L\hat{x};\hat{x},2L^{\beta}(\hat{z}-\hat{y}))/L 
+ \psi_\delta(Lx;x,0)/L 
+ C\eta 
+ \eta\sqrt{1+|x|^2} \\
&\leq CL^{-\beta} 
+ C\eta 
+ \eta\sqrt{1+|x|^2}.
\end{aligned}
\end{equation*}
Let $\eta \to 0$ then we obtain one side of the result. The other side follows from an analogous way and we omit it here.
\end{proof}

\begin{corollary}\thlabel{convergence rate infinity}
$|\lambda(L,p) - \Lambda(p)| \leq C(p)L^{\frac{-\tau}{2\tau+1}}$ for some $C(p) \in \RR$.
In the periodic case,
the rate improves to $L^{-\frac{1}{2}}$.
\end{corollary}

\begin{proof}
We have
\begin{equation*}
\begin{aligned}
|\lambda(L,p) - \Lambda(p)| 
&\leq &&|\lambda(L,p) - \epsilon v_p^L(0)| 
+ |\epsilon v_p^L(0) - \epsilon v_p(0)| 
+ |\epsilon v_p(0) - \Lambda(p)| \\
&\leq &&|\lambda(L,p) - \epsilon v_p^L(0)| 
+ |\epsilon v_p(0) - \Lambda(p)| 
+ C(p) L^{\frac{-\tau}{2\tau+1}}.
\end{aligned}
\end{equation*}
The conclusion follows by letting $\epsilon \to 0$.

In the periodic case
one readily verify that
$$\rho(R;\tilde{c}) = 0$$
for sufficiently large $R$. 
Hence there exists $C > 0$ such that
$$\rho(R;\tilde{c}) + \delta R \leq C \delta.$$
The same arguments as in the proof of \thref{convergence rate corrector} then yield
$$|\lambda(L,p) - \Lambda(p)| \leq C(p)L^{-\frac{1}{2}}.$$
\end{proof}

The proofs of \thref{infity limit} and \thref{infinityrate} follow from applying \thref{maintheoreminfinity} and \thref{convergence rate infinity} respectively
to the arguments established for \thref{zero limit} and \thref{zerorate}.

\section{Influence of the coefficients: The proof of the rest results}
This section presents the remaining proofs.

\begin{proof}
[\textbf{Proof of \thref{zeroexpression}}]
It follows from by taking $\psi \equiv 1$ as a test function in \eqref{zeroeigenvalues}.
\end{proof}

\begin{proof}
[\textbf{Proof of \thref{iota expression}}]
By the definition of $\iota$, consider the problem
$$a(x)D^2\psi + f(x) = \iota(f).$$
Multiplying both sides by $a^{-1}(x)$ and integrating over $[-R,R]$ gives
$$D\psi(R) - D\psi(-R) + \int_{-R}^{R}a^{-1}(x)f(x)dx = \iota(f)\int_{-R}^{R}a^{-1}(x)dx.$$
Dividing this equality by $2R$ and taking the limit as $R \to +\infty$ yields the result.
\end{proof}

\begin{lemma}
If $\tilde{c} = 0$, then 
\begin{equation}
\overline{H}(x,p) = \sum\limits_{i,j=1}^{n}a^{ij}(x)p_ip_j 
+ \sum\limits_{i=1}^{n}b^i(x)p_i 
+ c(x),
\end{equation}
where $\overline{H}$ is as defined by \eqref{transHamilton}.
\end{lemma}

\begin{proof}
Since $\psi \equiv 1$ is a solution of \eqref{transHamilton},
the conclusion follows directly from the definition of the effective Hamiltonian.
\end{proof}

Therefore $\Lambda(p)$ is the effective Hamiltonian for the problem
\begin{equation}\label{infiPrinEigen}
\sum\limits_{i,j=1}^{n}a^{ij}(D_iu-p_i)(D_ju-p_j) 
+ \sum\limits_{i=1}^{n}b^i(D_iu-p_i) 
+ c = \lambda
\end{equation}
when $\tilde{c} = 0$.
Moreover,
in one dimension, 
$\Lambda(p)$ could be expressed in terms of integral quantities of the coefficients.

\begin{proof}
[\textbf{Proof of \thref{infinity in 1-dim}}]
For any $\delta > 0$, let $u$ be a $\delta-$approximate corrector of \eqref{infiPrinEigen}.
Then
$$\Lambda(p) \geq \overline{H}(x,Du - p) - \delta = a(x)(Du - p + \frac{b}{2a}(x))^2 + c(x) - \frac{b^2}{4a}(x) - \delta.$$
Since the above inequality holds for all $x \in \RR$, 
it follows that $\Lambda(p) \geq M - \delta$,
where
$$M := \sup\limits_{x \in \RR}\{c(x) - \frac{b^2}{4a}(x)\}.$$ 
Letting $\delta \to 0$ gives $\Lambda(p) \geq M$ for all $p \in \RR$.

For all $\lambda \geq j_+(M)$ with $j_+$ defined as in \eqref{infirelation}, 
let 
\begin{equation}
u_+(x) = px - \int_{0}^{x}\frac{b(y)}{2a(y)}dy + \int_{0}^{x}\sqrt{\frac{\lambda - c(y)}{a(y)} + \frac{b^2(y)}{4a^2(y)}}dy,
\end{equation}
where $p = j_+(\lambda)$. 
Then $u_+$ is sublinear at infinity and solves
$$\overline{H}(x,Du_+ - p) = (j_+)^{-1}(p).$$
By the uniqueness of the effective Hamiltonian, we have $\Lambda(p) = (j_+)^{-1}(p)$. 
A symmetric argument applies for $p \leq j_-(M)$.
Finally, note that $\Lambda(j_{\pm}(M)) = M$.
The convexity of $\Lambda(p)$,
inherited from that of $\overline{H}$ in $p$ \cite{tran2021hamilton},
together with the bound $\Lambda(p) \geq M$ for all $p \in \RR$, 
implies that $\Lambda(p) = M$ for all $p \in [j_-(M), j_+(M)]$.  
\end{proof}

\begin{proof}
[\textbf{Proof of \thref{smaller eigenvalues}}]
Since $b$ is almost periodic with $<b> = 0$,
one has $M := \sup_{x \in \RR^n}\{c - \frac{b^2}{4a}(x)\} = c$.
Consequently $|j_{\pm}(M)| = \frac{<b^2>}{2a} \neq 0$,
where $j_{\pm}$ is defined in \eqref{infirelation}.
For any $|p| \leq |j_{\pm}(M)|$,
\thref{infinity in 1-dim} implies that $\Lambda(p) = c \leq ap^2 + c$.
For $|p| > |j_{\pm}(M)|$,
we have
\begin{equation*}
\sqrt{a}|p| 
= \Xint{-}\sqrt{\Lambda(p) - c + \frac{b^2}{4a}(x)}dx 
> \sqrt{\Lambda(p) - c},
\end{equation*}
which yields $\Lambda(p) < ap^2 + c$.
\end{proof}

\begin{proof}
[\textbf{Proof of \thref{slow-down speeds}}]
To prove $\omega(1;L) < 2\sqrt{ac}$ for large $L$,
we only need to prove that
\begin{equation*}
\inf_{p > 0} \frac{\Lambda(p)}{p} < 2\sqrt{ac}.
\end{equation*}
In fact,
we have
\begin{equation*}
\inf_{p > 0} \frac{\Lambda(p)}{p} 
\leq \frac{\Lambda(\sqrt{\frac{c}{a}})}{\sqrt{\frac{c}{a}}} 
< \frac{a (\sqrt{\frac{c}{a}})^2 + c}{\sqrt{\frac{c}{a}}}
= 2\sqrt{ac}.
\end{equation*}
The inequality $\omega(-1;L) < 2\sqrt{ac}$ for large $L$ follows from a parallel arguments.
\end{proof}

To prove \thref{accelerate}, 
we denote the quantities by $\overline{H}(x,q;\tilde{c})$, 
$\Lambda(p;\tilde{c})$ and $\lambda(p;\tilde{c})$ 
to emphasize their dependence on $\tilde{c}$.

\begin{proposition}
Assume $\tilde{c}$ satisfies $<\tilde{c}> = 0$, then
$$\Lambda(p,\tilde{c}) \geq \Lambda(p,0) \quad \hbox{ and } \quad \lambda(p;\tilde{c}) \geq \lambda(p;0).$$
\end{proposition}

\begin{proof}
For any $x \in \RR^n$ and $\delta > 0$,
let $\psi$ be a $\delta$-approximate corrector of \eqref{transHamilton}. 
Integrating the equation over $\RR^n$ and averaging yields
\begin{equation*}
\begin{aligned}
\overline{H}(x,q;\tilde{c}) + \delta 
\geq 
&\Xint{-} \sum\limits_{i,j=1}^{n}a^{ij}(x)D_{y_i}\psi(y)D_{y_j}\psi(y) dy  \\
&+ \sum\limits_{i,j=1}^{n}a^{ij}(x)q_iq_j 
+ \sum\limits_{i=1}^{n}b^i(x) q_i 
+ c(x) 
\geq \overline{H}(x,q;0).
\end{aligned}
\end{equation*}
Letting $\delta \to 0$ gives
$$\overline{H}(x,q;\tilde{c}) \geq \overline{H}(x,q;0).$$
Now suppose $\varphi$ satisfies
$$\|\overline{H}(\cdot, D\varphi - p; \tilde{c}) - \Lambda(p;\tilde{c})\| \leq \delta.$$
By \thref{characterization of effective Ham} one thus get that
$$\Lambda(p;\tilde{c}) + \delta \geq \overline{H}(x,D\varphi-p;\tilde{c}) \geq \overline{H}(x,D\varphi-p;0) \geq \Lambda(p,0)$$
for some $x \in \RR^n$.
The conclusion follows from letting $\delta \to 0$.

The proof of $\lambda(p;\cdot)$ is analogous.
\end{proof}

We will prove that the above inequalities are strict when $\tilde{c} \not \equiv 0$,
beginning with the one-dimensional case. 

\begin{theorem}\thlabel{positive when n=1}
Assume $a$ and $b$ are constants with $a > 0$,
and let $c \not \equiv 0$ is an almost periodic function with $<c> = 0$.
If $\lambda$ is the effective Hamiltonian of the following problem
\begin{equation}\label{positivity}
au'' + a(u')^2 + bu' + c = \lambda.
\end{equation}
Then $\lambda > 0$.
\end{theorem}

\begin{proof}
There exists $p_0 \in \RR$ such that $b = 2ap$,
therefore \eqref{positivity} can be rewritten as
\begin{equation*}
au'' + a(u'+p)^2 + c = \lambda + ap^2. 
\end{equation*}	
Let $\lambda(p)$ be the effective Hamiltonian of the following problem
\begin{equation}\label{principaleigenvalue}
H(x,u'+p,u'') = \lambda(p),
\end{equation}
where
\begin{equation*}
H(x,p,X) = aX + ap^2 + c(x).
\end{equation*}
It is therefore equivalent to prove that $\lambda(p) > ap^2$ for all $p \in \RR$.

According to \cite{liang2022propagation,berestycki2012spreading},
there exists $p_0 \geq 0$ such that for any $|p| > p_0$,
\eqref{principaleigenvalue} admits a sub-linear solution $u$ with $u'$ almost periodic,
and for $|p| \leq p_0$, $\overline\lambda(p) \equiv\overline \lambda(p_0)$.
When $|p| > p_0$, 
let $u$ be the sub-linear solution.
Integrating \eqref{principaleigenvalue} and averaging yields
\begin{equation*}
\lambda =  a <(u')^2>. 
\end{equation*}
Since $c$ is almost periodic and non-trivial,
it follows that $u' \not\equiv 0$.
The almost periodicity of $u'$ then implies $\lambda > 0$.

If $p_0 > 0$,
we have
$\lambda(p) \equiv \lambda(p_0) \geq ap_0^2 > ap^2$
for all $|p| < p_0$.
Furthermore,
since $\lambda(p)$ is smooth by \thref{smoothness of the effective Hamiltonian},
we have $\lambda(p_0) > ap_0^2$.

If $p_0 = 0$, 
by \cite{markus1956oscillation},
we have 
\begin{equation*}
\begin{aligned}
[\lambda(0), +\infty) = \{\lambda : &\hbox{ every non-trivial solution of }\\ &(au')' + c(x)u = \lambda u 
\hbox{ has at most one zero on } \RR \}.
\end{aligned}
\end{equation*}
Since this set is a proper subset of $\RR_+$,
it follows that $\lambda(0) > 0$.
\end{proof}

For the case of higher dimension,
we have

\begin{theorem}\thlabel{heterogenerous}
Let $A = (a^{ij})_{i,j=1}^n$ be a constant positive definite matrix
and $B = (b^i)_{i=1}^n$ a constant vector.
Suppose $c \not \equiv 0$ is almost periodic with $<c> = 0$
and let $\lambda$ be the effective Hamiltonian of the following problem
\begin{equation}\label{positiveeigenvalue}
\sum\limits_{i,j=1}^{n}a^{ij}D_{ij}u 
+ \sum\limits_{i,j=1}^{n}a^{ij}D_iuD_ju 
+ \sum\limits_{i=1}^{n}b^iD_iu + c = \lambda.
\end{equation}
Then $\lambda > 0$.
\end{theorem}

\begin{proof}
Write $\lambda = \lambda(c)$ to emphasize the dependence of $\lambda$ on $c$.
We will prove that $\lambda$ is convex in $c$.
In fact,
for any $r \in (0,1)$ and almost periodic functions $c_0, c_1$,
let $u_0,u_r,u_1$ be the unique bounded and uniformly continuous solutions of
\begin{equation*}
\sum\limits_{i,j=1}^{n}a^{ij}D_{ij}u_s 
+ \sum\limits_{i,j=1}^{n}a^{ij}D_iu_sD_ju_s 
+ \sum\limits_{i=1}^{n}b^iD_iu_s + c_s = \epsilon u_s, \quad
s=0,r,1
\end{equation*} 
respectively,
where $\epsilon > 0$
and $c_r = rc_0 + (1-r)c_1$.
Let $v = ru_0 + (1-r)u_1$,
then we have
\begin{equation*}
\begin{aligned}
&\sum\limits_{i,j=1}^{n}a^{ij}D_{ij}v 
+ \sum\limits_{i,j=1}^{n}a^{ij}D_ivD_jv 
+ \sum\limits_{i=1}^{n}b^iD_iv + c_r \\
&= r \epsilon u_0 + (1-r)\epsilon u_1 
- r(1-r)\sum\limits_{i,j=1}^{n}a^{ij}D_i(u_0-u_1)D_j(u_0-u_1)
\end{aligned}
\end{equation*}
Letting $\epsilon \to 0$ and applying \thref{characterization of effective Ham} yields
\begin{equation*}
\lambda(rc_0 + (1-r)c_1) \leq r \lambda(c_0) + (1-r)\lambda(c_1).
\end{equation*}

Now since $c$ is almost periodic with $<c> = 0$,
there exists a direction,
without loss of generality, take $e_1$,
such that
\begin{equation*}
\overline{c}(x_1) = \Xint{-}c(x_1,x_2,\cdots,x_n)dx_2\cdots dx_n 
\end{equation*}
is non-trivial with $<\overline{c}> = 0$.
This claim is established in \thref{directionselect}.

Now for any $\epsilon > 0$,
there exists $R > 0$ such that
\begin{equation*}
\|\overline{c}(x_1) - \frac{1}{(2R)^{n-1}}\int_{-R}^{R}\cdots \int_{-R}^{R}c(x_1,x_2,\cdots,x_n)dx_2 \cdots dx_n\| \leq \epsilon.
\end{equation*}
Note that Riemann integral could be approximated by Riemann sum,
that is,
\begin{equation*}
\begin{aligned}
\sum\limits_{i_2,\cdots, i_n=-m}^{m-1}c(x_1,x_2 + \frac{i_2R}{m},\cdots,\frac{i_nR}{m})(\frac{1}{2m})^{n-1} \to \frac{1}{(2R)^{n-1}}\int_{-R}^{R}\cdots \int_{-R}^{R}c(x_1,x_2,\cdots,x_n)dx_2 \cdots dx_n
\end{aligned}
\end{equation*}
as $m \to \infty$.
Furthermore,
since for any $y \in \RR^n$,
$\lambda(c(\cdot)) = \lambda (c(\cdot + y))$, 
and $\lambda$ is convex in $c$,
we have
\begin{equation*}
\lambda(c) \geq \lambda(\sum\limits_{i_2=-m}^{m-1}\cdots \sum\limits_{i_n=-m}^{m-1}c(x_1,x_2 + \frac{i_2R}{m},\cdots,\frac{i_nR}{m})(\frac{1}{2m})^{n-1}).
\end{equation*}
Letting $m \to \infty$ gives
\begin{equation*}
\lambda(c) \geq \lambda(\frac{1}{(2R)^{n-1}}\int_{-R}^{R}\cdots \int_{-R}^{R}c(x_1,x_2,\cdots,x_n)dx_2 \cdots dx_n),
\end{equation*}
which implies that
\begin{equation*}
\lambda(c) \geq \lambda(\overline{c}) - \epsilon.
\end{equation*}
Taking $\epsilon \to 0$ yields
\begin{equation*}
\lambda(c) \geq \lambda(\overline{c}).
\end{equation*}
For $\lambda(\overline{c})$,
the problem \eqref{positiveeigenvalue} reduces to
\begin{equation*}
a^{11}D_{11}u + a^{11}D_1uD_1u + b^1D_1u + \overline{c}(x_1) = \lambda(\overline{c}),
\end{equation*}
whose positivity is guaranteed by \thref{positive when n=1}.
\end{proof}

\begin{lemma}\thlabel{directionselect}
Assume $c$ is an almost periodic function with $<c> = 0$,
then there exists a direction $e \in \mathbb{S}^{n-1}$ such that
the function
\begin{equation*}
h(z) := \Xint{-}_{\{e \cdot x = z\}}c(x)dx
\end{equation*}
is a nontrivial almost periodic function on $\RR$ with $<h> = 0$.
\end{lemma}

\begin{proof}
Define
\begin{equation*}
c_\lambda :=  \Xint{-}c(x)e^{-i \lambda \cdot x}dx
\end{equation*}
and put $C := \{\lambda \in \RR^n : c_\lambda \neq 0\}$.
Then by \cite{bochner1935almost}, $C$ is at most countable and we will associate with $c$ a series,
write
\begin{equation*}
c \sim \sum_{\lambda \in C} c_\lambda e^{i \lambda \cdot x}. 
\end{equation*}
The numbers $c_\lambda$ are the Fourier coefficients of $c$ and $\lambda \in C$ the Fourier exponents.

Now,
let us fixed $e$ and consider the function
\begin{equation*}
h(z) = \Xint{-}_{\{e \cdot x = z\}}c(x)dx.
\end{equation*}
Then
for any $\xi \in \RR$,
\begin{equation*}
\begin{aligned}
h_\xi := \Xint{-}h(z)e^{-i \xi z}dz
&= \lim\limits_{R \to \infty}\int_{-R}^{R}\Xint{-}_{e \cdot x = z}c(x)dSe^{-i \xi z}dz \\
&= \Xint{-}c(x) e^{-i \xi e \cdot x}dx = c_{\xi e},
\end{aligned}
\end{equation*}
where we have used Fubini's theorem.
Therefore,
we have
\begin{equation*}
h(z) \sim \sum\limits_{\lambda \in \Lambda(e)} c_\lambda e^{i (\lambda \cdot e)z},
\end{equation*}
where 
\begin{equation*}
\Lambda(e) = \{\lambda \in \RR^n : \lambda = t e \hbox{ for some } t \in \RR\}.
\end{equation*}
Therefore we only need to select $e = \lambda/\|\lambda\|$ for some $\lambda \in C$ (note $C$ is nonempty and $0 \notin C$ by assumption),
then the corresponding $h(z)$ is nonzero.
The almost periodicity and mean-zero property of $h$ are inherited from $c$. 
\end{proof}

As a conclusion, 
we have
\begin{theorem}
Assume $\tilde{c} \not \equiv 0$ satisfies $<\tilde{c}> = 0$, then
\begin{equation*}
\omega(e;+\infty,\tilde{c}) > \omega(e;+\infty,0), \quad 
\omega(e;0,\tilde{c}) > \omega(e;0,0).
\end{equation*}
\end{theorem} 

\begin{proof}
The latter inequality follows from the fact $\lambda(p;\tilde{c}) > \lambda(p;0)$ for all $p \in \RR^n$,
which is a direct consequence of \thref{heterogenerous}.
For the first one,
we only need to show that for any $R > 0$,
there exists $\epsilon > 0$ such that $\overline{H}(x,p;\tilde{c}) > \overline{H}(x,p;0) + \epsilon$ for $(x,p) \in \RR^n \times B(0,R)$.

Note that for any $(x,p) \in \RR^n \times B(0,R)$,
we have $\overline{H}(x,p;\tilde{c}) > \overline{H}(x,p;0)$ by \thref{heterogenerous}.
Assume by contradiction that there exists $\{x_n\}_{n=1}^{\infty} \subset \RR^n$ and 
$\{p_n\}_{n=1}^{\infty} \subset B(0,R)$ such that
\begin{equation*}
\lim\limits_{n \to \infty}(\overline{H}(x_n,p_n;\tilde{c}) - \overline{H}(x_n,p_n;0)) = 0.
\end{equation*}
Since $A,B$ and $c$ are almost periodic,
without loss of generality,
suppose that $A(x_n) \to \overline{A}$,
$B(x_n) \to \overline{B}$
and $c(x_n) \to \overline{c}$ as $n \to \infty$.
Assume $p_n \to p$ as $n \to \infty$,
then $\lim\limits_{n \to \infty}\overline{H}(x_n,p_n;\tilde{c})$ is the effective Hamiltonian of
\begin{equation*}
\sum\limits_{i,j=1}^{n}\overline{a}^{ij}D_{ij}u(y) 
+ \sum\limits_{i,j=1}^{n}\overline{a}^{ij}(D_iu(y)+p_i)(D_ju(y)+p_j) 
+ \sum\limits_{i=1}^{n}\overline{b}^i(D_iu(y)+p_i)+\overline{c} + \tilde{c}(y) = \lambda,
\end{equation*}
and $\lim\limits_{n \to \infty}\overline{H}(x_n,p_n;0) = \overline{a}^{ij}p_ip_j + \overline{b}^ip_i + \overline{c}$.
By \thref{heterogenerous},
we have
\begin{equation*}
\lim\limits_{n \to \infty}\overline{H}(x_n,p_n;\tilde{c}) > \lim\limits_{n \to \infty}\overline{H}(x_n,p_n;0),
\end{equation*}
a contradiction.
\end{proof}

%%%%%%%%%%%%%%%%%%%%%%%%%%%%%%%%%%%%%%%%%%%%%%%%%%%%%%%%%%%%%%%%
%conference
%%%%%%%%%%%%%%%%%%%%%%%%%%%%%%%%%%%%%%%%%%%%%%%%%%%%%%%%%%%%%%%%
%\small
%\begin{thebibliography}{99}
%	\setlength{\parskip}{0pt} 
%	\bibitem{}
%\end{thebibliography}

\bibliographystyle{plain}
\bibliography{reference.bib}
%%%%%%%%%%%%%%%%%%%%%%%%%%%%%%%%%%%%%%%%%%%%%%%%%%%%%%%%%%%%%%%%
%ending
%%%%%%%%%%%%%%%%%%%%%%%%%%%%%%%%%%%%%%%%%%%%%%%%%%%%%%%%%%%%%%%%
\clearpage
\end{document}